\documentclass[11pt]{article}
\usepackage{graphicx}
\usepackage{amsmath,amsfonts,amsthm}
\usepackage{amssymb,latexsym}
\usepackage{fixmath}
\usepackage{mathrsfs,amsbsy}
\usepackage{dsfont}
\usepackage{enumerate}

\def\wtd{\widetilde}
\def\what{\widehat}

\DeclareMathOperator{\diag}{diag}

\DeclareMathOperator{\sign}{sign}

\DeclareMathOperator{\HH}{H}
\DeclareMathOperator{\T}{T}

\def\ba{\pmb{a}}
\def\bb{\pmb{b}}
\def\bc{\pmb{c}}

\def\be{\pmb{e}}
\def\Bf{\pmb{f}}

\def\bp{\pmb{p}}
\def\bq{\pmb{q}}

\def\bw{\pmb{w}}
\def\bx{\pmb{x}}

\def\bz{\pmb{z}}

\renewcommand{\Re}{\mbox{\sc Re\,}}
\renewcommand{\Im}{\mbox{\sc Im\,}}

\newtheorem{proposition}{Proposition}[section]
\newtheorem{theorem}{Theorem}[section]
\newtheorem{lemma}{Lemma}[section]

\theoremstyle{definition}

\newtheorem{remark}{Remark}[section]
\newtheorem{example}{Example}[section]

\numberwithin{equation}{section}
\numberwithin{figure}{section}
\numberwithin{table}{section}

\allowdisplaybreaks

\usepackage{kbordermatrix}

\usepackage{caption}
\usepackage{amsfonts,amsthm}
\usepackage{amsmath,amssymb,enumerate,color}
\usepackage{algorithm,algorithmic}
\usepackage{epsfig,graphicx,psfrag,mathrsfs,subfigure}
\usepackage{eucal}
\usepackage{yhmath}
\usepackage{hyperref}
\usepackage{diagbox}
\RequirePackage{multirow}
 
\usepackage{textcomp}
\usepackage[]{fontenc}
\usepackage{extarrows}
\def\wtd{\widetilde}
\def\what{\widehat}
\usepackage{rotating}
\usepackage{lscape}
\usepackage{bm}
\usepackage{xspace}
\usepackage{tikz}
\usetikzlibrary{petri} 
\usetikzlibrary{shadows,arrows,positioning,shapes.geometric}  
\usepackage{geometry}
\usepackage{lipsum}

\def\ba{\pmb{a}}
\def\bb{\pmb{b}}
\def\bc{\pmb{c}}

\def\be{\pmb{e}}
\def\Bf{\pmb{f}}

\def\bp{\pmb{p}}
\def\bq{\pmb{q}}

\def\bw{\pmb{w}}
\def\bx{\pmb{x}}

\def\bz{\pmb{z}}
\def\bzs{\mathbf{0}}
\def\tu{\mathfrak{u}}

\def\diag{{\rm diag}}

\def\scrR{\mathscr{R}}

\def\wtd{\widetilde}
\def\what{\widehat}

\def\bbC{\mathbb{C}}

\def\bbP{\mathbb{P}}
\def\bbR{\mathbb{R}}

\renewcommand{\algorithmicrequire}{\textbf{Input:}}
\renewcommand{\algorithmicensure}{\textbf{Output:}}
 
\numberwithin{equation}{section}
\numberwithin{figure}{section}
\numberwithin{table}{section}

\graphicspath{{figures/}{figure/}{pictures/}%
	{picture/}{pic/}{pics/}{image/}{images/}}
 
\title{Interpolation constrained rational minimax approximation with barycentric representation}
\author{Lei-Hong Zhang\thanks{Corresponding author. School of Mathematical Sciences, Soochow University, Suzhou 215006, Jiangsu, China. This work was
 supported in part by the National Natural Science Foundation of China (NSFC-12471356, NSFC-12371380), Jiangsu Shuangchuang Project (JSSCTD202209), and Academic Degree and Postgraduate Education Reform Project of Jiangsu Province.
        Email: {\tt longzlh@suda.edu.cn}.} \and Ya-Nan Zhang\thanks{School of Mathematical Sciences, Soochow University, Suzhou 215006, Jiangsu, China. Email: {\tt ynzhang@suda.edu.cn}.} 
        }
  
\date{ } 
\begin{document}

\maketitle

\begin{abstract}
In this paper, we propose a novel  dual-based Lawson's method, termed {\sf b-d-Lawson},  designed for addressing  the rational minimax approximation under specific interpolation conditions. The {\sf b-d-Lawson} approach incorporates two pivotal components that have been recently gained prominence in the realm of the rational approximations: the {\tt b}arycentric representation of the rational function and the {\tt d}ual framework for tackling minimax approximation challenges. The employment of  barycentric formulae enables a streamlined parameterization of the rational function, ensuring natural satisfaction of interpolation conditions while mitigating numerical instability typically associated with Vandermonde basis matrices when monomial bases are utilized. This   enhances both the accuracy and computational stability of the method. To address the bi-level min-max structure, the dual framework effectively transforms the challenge into a max-min dual problem, thereby facilitating the efficient application of Lawson's iteration. The integration of this dual perspective is crucial for optimizing the approximation process. We will discuss several applications of interpolation-constrained rational minimax approximation and illustrate numerical results to evaluate the performance of the {\sf b-d-Lawson} method. 
\end{abstract}

\medskip


\section{Introduction}\label{sec_intro}

Computing the minimax (also known as best or Chebyshev) polynomial and/or rational approximations for a given complex or real function $f$ is a classical problem in approximation theory \cite{tref:2019a}. Let $\bbP_n$ be the set of complex polynomials with degree $\le n$ and  $\scrR_{(n_1,n_2)}:=\{{p}/{q}|p\in \bbP_{n_1},~0\not\equiv q\in\bbP_{n_2}\}$. We call   $\xi\in \scrR_{(n_1,n_2)}$ a rational function of type $(n_1,n_2)$, and  conveniently, we denote  $\bbP_n=\scrR_{(n,0)}$ and $\scrR_{(n)}=\scrR_{(n,n)}$. Given a compact set $\cal {X}\subset \bbC$ in the complex plane,  the traditional minimax approximation is 
\begin{equation}\label{eq:bestf0}
\inf_{\xi \in\scrR_{(n_1,n_2)}}\|f-\xi\|_{\infty},
\end{equation}
where $\|f-\xi\|_{\infty}$ stands for the  Chebyshev norm defined on  ${\cal X}$ as 
$\|f-\xi\|_{\infty}=\max_{x\in {\cal X}}|f(x)-\xi(x)|$. The set $\cal {X}\subset \bbC$ can be an interval (or the union of intervals) in the real line, a domain in $\bbC$ enclosed by Jordan curves, or a set containing discrete nodes in $\bbC$.

Interpolation constrained minimax polynomial/rational approximations are variants of the classical \eqref{eq:bestf0} and can arise in many situations. For instance, as the closed form solution $\xi$ of \eqref{eq:bestf0} is generally hard, numerical algorithms for \eqref{eq:bestf0} usually are designed based on sampled data $(x_j,f_j)$ from $f$ (for example the rational approximation of the Riemman zeta function \cite{nast:2023}). In this case, except for certain samples $(t_j,y_j)$ with $y_j=f(t_j)$ for $1\le j\le \ell$, errors for $f_j\approx f(x_j)$ will be introduced in other data. In finite precision, this is case for almost all functions $f$ numerically.  Given {\it interpolation} data $(t_j,y_j)$ for $1\le j\le \ell$ and {\it sampled} data $(x_j,f_j)$ for $1\le j\le m$, a general interpolation constrained minimax polynomial/rational approximation can then be put as 
\begin{equation}\label{eq:baryInterpMinmax0}
\inf_{\begin{subarray}{c} \xi\in \scrR_{(n_1,n_2)}\\
            \xi(t_j)=y_j, ~1\le j\le \ell\end{subarray}
            }\|\xi(\bx)-\Bf\|_\infty,~~{\rm where}~~\|\xi(\bx)-\Bf\|_\infty=\max_{1\le j\le m}|\xi(x_j)-f_j|,
\end{equation}
$\Bf=[f_1,\dots,f_m]^{\T}\in \bbC^m$ and $\bx=[x_1,\dots,x_m]^{\T}\in \bbC^m$.

In the literature, the interpolation constrained polynomial/rational minimax approximations have been discussed in, e.g., \cite{cops:2002,fifr:1990,kali:2008,lixu:1990,lixu:1990b,neva:1919,pick1916,tawi:1974,wals:1932}. 
As a special case of \eqref{eq:baryInterpMinmax0}, the interpolation constrained minimax polynomial approximation arises in the convergence analysis of the Krylov subspace methods for the linear system $A\bz=\bb$. It is known that the accuracy of the $k^{\rm th}$ iteration $\bz_k^{\rm cg}$ in approximating $\bz^*=A^{-1}\bb\in \bbR^n$ from the  conjugate gradient iteration (CG) is governed by (see e.g., \cite[Chapter 6.11]{saad:2003})
$\|\bz_k^{\rm cg}-\bz^*\|_A\le\epsilon_k   \cdot\|\bz_0^{\rm cg}-\bz^*\|_A,$
where 
\begin{equation}\label{eq:minimaxployInt}
\epsilon_k=\min_{p\in \bbP_k,~p(0)=1~}\max_{1\le j\le n}|p(\lambda_j)|
\end{equation}
and $\lambda_j>0$ are eigenvalues of the positive definite matrix $A\in \bbR^{n\times n}$.
Analogously, for the GMRES method with  $A=X\diag(\lambda_1,\dots,\lambda_n)X^{-1}$ and $\lambda_j\in \bbC$, the $k^{\rm th}$ iteration $\bz_k^{\rm gmres}$  gives the bound  (see e.g., \cite[Chapter 6.11]{saad:2003} and \cite{fifr:1990})
$\|\bb-A\bz_k^{\rm gmres}\|_2\le\epsilon_k \cdot  \kappa_2(X) \cdot\|\bb-A\bz_0^{\rm gmres}\|_2,$
where $\kappa_2(X)=\|X\|_2 \|X^{-1}\|_2$ is the condition number of $X$ and $\epsilon_k$ is defined \eqref{eq:minimaxployInt}. 

For the rational case, the real and continuous version of \eqref{eq:baryInterpMinmax0} are discussed in {\cite{ehle:1976,lixu:1990,lixu:1990b}, \cite[Section 2.2]{wang:1980}  and \cite[Section 6.2]{wazh:2004},  where ${\cal X}=[a,b]$ is an interval. Besides the interpolation conditions $\xi(t_j)=y_j
~(1\le j\le \ell)$, high-order derivative requirements $\xi^{(i)}(t_j)=f^{(i)}(t_j), ~j=1,\dots,\ell,~i=0,\dots,j_i$ are also imposed}. Conditions for the existence, uniqueness and some characterizations of such minimax rational approximations are established.

Complex versions of \cite{ehle:1976,lixu:1990,lixu:1990b} can be found in, e.g., \cite{lubr:2011,neva:1919,pick1916}. As a special and important case, indeed, the Nevanlinna-Pick interpolation  \cite{pick1916,neva:1919} is an important problem in complex analysis and operator theory, which describes a particular  type of complex interpolation problem. It  was originally studied by Pick in 1916 \cite{pick1916}  and independently by Nevanlinna in 1919 \cite{neva:1919}. Let ${\bf H}^\infty$ be the space of bounded analytic functions in the unit disc ${\cal D}=\{z\in \bbC: |z|<1\}$.  The finite case of Nevanlinna-Pick interpolation problem finds a holomorphic map $\xi: {\cal D} \rightarrow {\cal D}$ such that $\xi(t_j)=y_j, ~(1\le j\le \ell)$, where   $\{t_1,\dots,t_\ell\}\subset {\cal D}$ and values $\{y_1,\cdots, y_\ell\}\subset{\cal D}$ are given.
 {Equivalently, we 
 $$
 {\rm find }~ \xi\in  {\bf H}^\infty, ~ \|\xi\|_\infty \le  1~ {\rm with~}  \xi(t_i) = y_i, ~i = 1,\dots, \ell,
 $$
 where  $\|\xi\|_\infty$ is the essential supremum of $|\xi(z)|$ on the unit circle (a consequence of the maximum modulus principle).}
 Herglotz et. al. \cite{hesp:1991} proved that  the Nevanlinna-Pick interpolation exists and the minimal norm solution  $\xi$ is unique which is also rational with the degrees of both numerator and denominator  less than $\ell$ (see also \cite{alyo:1983,lubr:2011}). Hence, computing the minimal  Chebyshev norm solution is amount to finding the minimax rational approximation in ${\bf H}^\infty$ subject to the interpolation conditions $\xi(t_j)=y_j~(1\le j\le \ell)$ (see \cite[Problem I]{alyo:1983}).  Applications of the Nevanlinna-Pick interpolation problem include the sensitivity minimization \cite[Equ. (2.1)]{bygl:2001} and the maximal power transfer \cite[Section II]{bygl:2001}   arising from systems and control.

With the setting of \eqref{eq:baryInterpMinmax0},  in this paper, we introduce  a numerical method, referred to as {\sf b-d-Lawson},  for addressing interpolation constrained minimax  rational approximations. We remark that  {\sf b-d-Lawson} specifically targets scenarios where the interpolated points represent a small,  carefully selected  subset of sampling points that require exact interpolation without outliers; it reduces to a method for the classical rational minimax approximation when no interpolation conditions are imposed. Our treatment is based on two key points that have been recently gained prominence in the realm of the rational approximations. The first crucial component is the barycentric formula to represent the target rational $\xi$ in \eqref{eq:baryInterpMinmax0}. We will see that  the use of the barycentric formula facilitates us to match the interpolation conditions $\xi(t_j)=y_j~(1\le j\le \ell)$ naturally, and on the other hand, alleviates the numerical instability \cite{brnt:2021,zhsl:2024} caused by the Vandermonde basis matrices from using the monomial bases.  The second key component  is the dual framework \cite{yazz:2023,zhan:2026,zhyy:2025,zhzz:2025} for the minimax problem and the design of proper Lawson's iteration. The recent developments in \cite{yazz:2023,zhyy:2025} reformulate the minimax approximation into the (max-min) dual problem and explain that the classical Lawson iteration is an ascent convergent method for solving the dual problem. 
A new dual-based Lawson iteration \cite{zhyy:2025}, termed   {\sf d-Lawson}, was proposed from this dual framework and its convergence analysis has been discussed in \cite{zhha:2025}. Due to the incorporation of these two key components, namely, the barycentric representation and the dual framework, the new method has been aptly named {\sf b-d-Lawson}.

 The structure of this paper is as follows: In Section \ref{sec:baryIntminmax}, we will introduce the barycentric formula as a representation for the rational functions and parametrize the feasible solution of \eqref{eq:baryInterpMinmax0} using the interpolation data $\{(t_j,y_j)\}_{j=1}^\ell$. With the barycentric representation, we will establish the   max-min structured dual problem of \eqref{eq:baryInterpMinmax0} in Section \ref{sec:dual}, and provide the detailed computational procedure for the dual function in Section \ref{subsec:dualfun}. The duality theory will be discussed in Section \ref{subsec:duality}; particularly,   weak duality ensures that any   dual function value provides a lower bound for the infimum  of \eqref{eq:baryInterpMinmax0}, while   strong duality generally says that we can solve the solution of \eqref{eq:baryInterpMinmax0} from the dual problem.  We shall show that   weak duality always holds but   strong duality is true under certain conditions. Based on the dual problem, the method of {\sf b-d-Lawson} will be introduced in Section \ref{sec:bdlawson}, where Lawson's idea is used to update the dual variables. Numerical results of {\sf b-d-Lawson} from several examples  will be reported in Section \ref{sec:numerical} and concluding remarks are drawn in Section \ref{sec:conclusion}.

 \section{The interpolation constrained rational minimax approximation}\label{sec:baryIntminmax}
\subsection{Barycentric representation}
The barycentric formula was originally used in the polynomial interpolation; Berrut and Trefethen provided a comprehensive  review for the barycentric Lagrange interpolation \cite{betr:2004}. Let  $\{(t_j,y_j)\}_{j=1}^{\ell}$ be the given interpolation data and $\{t_j\}_{j=1}^{\ell}$ be distinct. The unique interpolation polynomial $p\in \bbP_{\ell-1}$ can be expressed by 
\begin{equation}\label{eq:barypoly}
p(x)= {\sum_{j=1}^\ell\frac{\beta_j y_j}{x-  t_j}}\bigg/{\sum_{j=1}^\ell\frac{\beta_j}{x-t_j}},~~{\rm with}~\beta_j=\frac{1}{\prod_{k\ne j}(t_j-t_k)},~1\le j\le \ell.
\end{equation}
This is the  barycentric formula (which is called the ``second (true) form of the barycentric formula" by Rutishauser \cite{ruti:1976}); see also \cite{betr:2004,henr:1962,henr:1979,ruti:1976,tayl:1945}.  For  $\beta_j$, 
\cite{wern:1984} considers  the computation of these weights in $(\ell-1)^2/2$ flops, and points out its interpolation property to the  rational case. Indeed, Schneider and Werner \cite{scwe:1986} noticed that for any set of non-zero weights $\{\beta_j\}_{j=1}^{\ell}$, the function \eqref{eq:barypoly} is a rational interpolant\footnote{{Indeed, for any $\beta_j\ne 0 ~(1\le j\le \ell)$, by $\lim_{x\rightarrow t_j}\frac{\sum_{j=1}^\ell\frac{\beta_jy_j}{x-t_j}}{\sum_{j=1}^\ell\frac{\beta_j}{x-t_j}}=y_j$, the barycentric formulation \eqref{eq:barypoly} $p(x)\in \scrR_{(\ell-1)}$ satisfies the interpolation conditions $p(t_j)=y_j$ for $j=1,\dots,\ell$.}}   of degree at most $\ell-1$. 

In \eqref{eq:barypoly},  we notice that both the numerator and denominator are represented by the singular basis $\{\frac{1}{x-t_j}\}_{j=1}^\ell$ which vary in size just algebraically, instead of the monomial basis $\{x^j\}_{j=0}^{\ell-1}$ which may vary exponentially (see \cite[Section 2.2]{fint:2018}).  Backward and forward error analysis of evaluating the interpolant $p(x)$ is established in \cite{high:2004a}, indicating that the barycentric formula \eqref{eq:barypoly}  is  forward stable  provided that  the associated Lebesgue constant\footnote{For the case $-1\le t_i\le 1$, the Lebesgue constant $\Lambda_{\ell-1}=\sup_{x\in [-1,1]}\sum_{j=1}^\ell|\ell_j(x)|$, where $\ell_j(x)=\frac{\prod_{k\ne j}(x-t_k)}{\prod_{k\ne j}(t_j-t_k)}$ (see \cite[Chapter 2]{chli:2000}). } $\Lambda_{\ell-1}$ is small.

The extension of the barycentric interpolation polynomial of \eqref{eq:barypoly} to the rational case has been  made in, e.g., \cite{angp:2025,fint:2018,flho:2007,henr:1979,ioni:2013,limp:2022,luws:2020,nase:2018,natr:2020,scwe:1986};  particularly, recent works  in \cite{fint:2018,nase:2018,natr:2020} {discuss} the non-interpolation mode \begin{equation}\label{eq:barycentricR}
\xi(x)= {\sum_{j=1}^\ell\frac{\alpha_j}{x-t_j}}\bigg/{\sum_{j=1}^\ell\frac{\beta_j}{{x}-t_j}}\in \scrR_{(\ell-1)},
\end{equation}
where $\{\alpha_j\}_{j=1}^\ell$, $\{\beta_j\}_{j=1}^\ell$ are coefficients and $\{t_j\}_{j=1}^\ell$ are {\it support points}. Different from the using of monomial bases, the resulting basis matrices for the numerator and denominator are the Cauchy matrices which generally  have better condition numbers than  the Vandermonde matrices \cite{beck:2000,high:2002,li:2008b,pan:2016,tref:2019a}. Regarding the evaluation of  $\xi(x)$, the rational approximation in the barycentric representation \eqref{eq:barycentricR} is backward stable (in the sense of the perturbation of the coefficients $\{\alpha_j\}_{j=1}^\ell$ and $\{\beta_j\}_{j=1}^\ell$) and can be forward stable when the support points  $\{t_j\}_{j=1}^\ell$ are suitably chosen in evaluation of not small $|\xi(x)|$ (see \cite{fint:2018,nase:2018,natr:2020} for more details). Moreover, the barycentric formula \eqref{eq:barycentricR} provides a parameterization for $\scrR_{(\ell-1)}$:  
\begin{theorem}{\rm (\cite[Theorem 3.1]{natr:2020})}\label{thm:barycentric}
Let $\{t_j\}_{j=1}^{\ell}, t_j\in \bbC$ be $\ell$ distinct nodes. As $\alpha_1,\dots, \alpha_{\ell}$ and $\beta_1,\dots,\beta_{\ell}$ range over all complex values, with at least one $\beta_j$ being nonzero, the functions \eqref{eq:barycentricR} 
range over the set of all rational functions of $\scrR_{(\ell-1)}$.
\end{theorem}

 \subsection{The interpolation constrained rational minimax approximation}
\label{sec:RminimaxInterp}
Now, we consider the problem \eqref{eq:baryInterpMinmax0} with $n=n_1=n_2.$
Let $\{(t_j,y_j)\}_{j=1}^\ell$ and $0\le \ell\le n+1$ be the given data to be interpolated by a rational approximation $\xi\in \scrR_{(n)}$ where  $t_j\in {\cal I}:=\{t_i\}_{i=1}^\ell \subset \bbC$ for $1\le j\le \ell$ are distinct. Also, we can choose other $n-\ell+1$ nodes $\{t_j\}_{j=\ell+1}^n$ in certain ways (see Section \ref{sec:bdlawson}) to form the overall $n+1$ distinct {\it support points} ${\cal T}:=\{t_j\}_{j=1}^{n+1}={\cal I}\cup {\cal \what I}$, where ${\cal \what I}=\{t_j\}_{j=\ell+1}^n$. Thus, we have the following parameterization 
\begin{equation}\label{eq:baryform}
\xi(x)=\frac{p(x)}{q(x)}:=\frac{\sum_{j=1}^\ell\frac{\beta_jy_j}{x-t_j}+\sum_{j=\ell+1}^{n+1}\frac{\alpha_j}{x-t_j}}{\sum_{j=1}^\ell\frac{\beta_j}{x-t_j}+\sum_{j=\ell+1}^{n+1}\frac{\beta_j}{x-t_j}}.
\end{equation}
It can be seen \cite{bemi:1997,wern:1984} that 
\begin{equation}\label{eq:interpcond}
\xi(t_j)=\lim_{x\rightarrow t_j}\xi(x)=y_j,~~\forall 1\le j\le \ell ~ {\rm with}~\beta_j\ne 0, 
\end{equation}
i.e., $\xi(x)$ given in \eqref{eq:baryform}  fulfills the interpolation conditions \eqref{eq:interpcond}. Conversely, for any $r\in \scrR_{(n)}$ satisfying the interpolation conditions  $r(t_j)=y_j~(1\le j\le \ell)$,  the following   theorem  guarantees that $r$ can be parametrized by \eqref{eq:baryform}. This is a generalization of  $\ell=0$ (Theorem \ref{thm:barycentric}) and $\ell=n+1$  (\cite[Lemma 2.1]{bemi:1997b}).

\begin{theorem}
\label{lem:baryrational2}
 Suppose $\{t_j\}_{j=1}^{n+1}, t_j\in \bbC$ are $n+1$ distinct nodes. For $0\le \ell\le n+1$, let $\{y_j\}_{j=1}^{\ell}, y_j\in \bbC$. Then any $r\in \scrR_{(n)}$ satisfying    $r(t_j)=y_j~(1\le j\le \ell)$  can be written as  the barycentric form \eqref{eq:baryform}, for some weights $\{\alpha_j\}_{j=\ell+1}^{n+1}$ and  $\{\beta_j\}_{j=1}^{n+1}$; in particular, if $t_j ~(\ell+1\le j\le n+1)$ is a pole of $r$, then $\beta_j=0$.
\end{theorem}
\begin{proof}
The proof is the same as that of \cite[Lemma 2.1]{bemi:1997b} in which $\ell=n+1$. Also, Theorem \ref{thm:barycentric} corresponds to $\ell=0$. The fact that $t_j$ is a pole of $r$  then  $\beta_j=0$ has already been given in \cite{bemi:1997} for the classical rational interpolation with $\ell=n+1$. 
\end{proof}

For the general barycentric formula \eqref{eq:baryform} satisfying the $1\le \ell\le n+1$ interpolation conditions $\xi(t_j)=y_j\in\bbC$ for $1\le j\le \ell$, we consider to compute a rational $\xi(x)$ that achieves the minimax approximation  (or, equivalently, parameters  $\{\alpha_j\}_{j=\ell+1}^{n+1}$ and  $\{\beta_j\}_{j=1}^{n+1}$ with $\beta_j\ne 0~(1\le j\le \ell)$ in \eqref{eq:baryform})
\begin{equation}\label{eq:baryInterpMinmax}
\inf_{\begin{subarray}{c} \xi\in \scrR_{(n)}\\
            \xi(t_j)=y_j, ~1\le j\le \ell\end{subarray}
            }\|\xi(\bx)-\Bf\|_\infty,~~{\rm where}~~\|\xi(\bx)-\Bf\|_\infty=\max_{1\le j\le m}|\xi(x_j)-f_j|
\end{equation}
over sample data $\{(x_j,f_j)\}_{j=1}^m~(m\ge 2n+2-\ell)$. Let  ${\cal X}:=\{x_j\}_{j=1}^{m}$ hereafter. In this setting, we assume, without loss of generalization, that ${\cal X}\cap {\cal T}=\emptyset$.  

Whenever the infimum of \eqref{eq:baryInterpMinmax} is attainable by the optimal parameters $\{\alpha_j^*\}_{j=\ell+1}^{n+1}$ and  $\{\beta_j^*\}_{j=1}^{n+1}$, we shall relate them with the numerator $p^*$ and the denomimator $q^*$ {\it via} 
\begin{equation}\label{eq:optimalpq}
p^*(x)=\sum_{j=1}^\ell\frac{\beta_j^*y_j}{x-t_j}+\sum_{j=\ell+1}^{n+1}\frac{\alpha_j^*}{x-t_j},~~q^*(x)=\sum_{j=1}^\ell\frac{\beta_j^*}{x-t_j}+\sum_{j=\ell+1}^{n+1}\frac{\beta_j^*}{x-t_j}.
\end{equation}

\begin{remark}\label{rkm:baryInterpMinmax}
We make a remark on two special cases with $\ell=0$ and $\ell=n+1$.
\item[i)]  The minimax approximation \eqref{eq:baryInterpMinmax} with $\ell=0$ is just the classical discrete rational minimax approximation:
\begin{equation}\label{eq:bestf}
 \inf_{\xi \in\scrR_{(n)}}\|\Bf-\xi(\bx)\|_{\infty},
\end{equation}
for which extensive discussions have been made  traditionally in e.g., \cite{bapr:1972,chlo:1963,elli:1978,guse:1999,gutk:1983,loeb:1957,rutt:1985,sava:1977,sako:1963,this:1993,will:1972,will:1979,wulb:1980} and     recently  in e.g., \cite{begu:2017,drnt:2024,fint:2018,gogu:2021,hoka:2020,nase:2018,natr:2020,yazz:2023,zhha:2025,zhyy:2025}. 

\item[ii)] For another special case with $\ell=n+1$, the problem \eqref{eq:baryInterpMinmax} reduces to finding $\{\beta_j\}_{j=1}^{n+1}$ in 
$\xi(x)=  {\sum_{j=1}^{n+1}\frac{\beta_jy_j}{x-t_j}}\big/{\sum_{j=1}^{n+1}\frac{\beta_j}{x-t_j}}$
 to achieve the minimax approximation over ${\cal X}$. This is related with the rational interpolation  \cite{bemi:1997b,betr:2004,wuyt:1974} in the literature. In this situation, the minimax approximation \eqref{eq:baryInterpMinmax} offers a way to determine good weights $\beta_j$ using data $\{(x_j,f_j)\}_{j=1}^m$ for the rational interpolation. Further discussions on the topic of the rational interpolation can be found in  e.g., \cite{bekl:2014,flho:2007,gukl:2012,ioni:2013,scwe:1986}.
\end{remark}

\section{The dual problem}\label{sec:dual}
 
With the parameterization $\xi(x)$ given in \eqref{eq:baryform}, we now consider the minimax approximation \eqref{eq:baryInterpMinmax}. Similar to the treatment for the classical polynomial \cite{yazz:2023} and rational minimax approximations \cite{zhha:2025,zhyy:2025}, {we introduce a real variable $
\eta= \max_{1\le j\le m}|\xi(x_j)-f_j|^2
$
which implies then that $ |\xi(x_j)-f_j|^2\le \eta,~ (1\le j\le m)$. 
Thus, we can   formulate the minimax approximation \eqref{eq:baryInterpMinmax} as a single-level minimization
\begin{equation}\label{eq:linearminmax}
\inf_{\eta, \{\alpha_j\},\{\beta_j\}} \eta, ~~s.t., |\xi(x_j)-f_j|^2\le \eta,~ (1\le j\le m).
\end{equation}
\cite[Theorem 2.4]{zhha:2025} establishes the precise relationship between the formulations \eqref{eq:linearminmax} and \eqref{eq:baryInterpMinmax} when $\ell=0$.}
For any barycentric representation $\xi(x)$ of \eqref{eq:baryform} with $q(x_j)\ne 0 ~(1\le j\le m)$,  reformulate the constraints $|\xi(x_j)-f_j|^2\le \eta ~ (1\le j\le m)$ as
$$
|p(x_j)-f_jq(x_j)|^2 \le \eta  |q(x_j)|^2,~  1\le j\le m,
$$ 
leading to an alternative form of \eqref{eq:linearminmax}:
\begin{equation}\label{eq:linearminmax2}
\inf_{\eta, \{\alpha_j\}, \{\beta_j\}} \eta, ~~s.t., |p(x_j)-f_jq(x_j)|^2\le \eta |q(x_j)|^2,~ 1\le j\le m.
\end{equation}
Applying Lagrange duality theory (see e.g., \cite{boyd:2004,nowr:2006}), one then can define the Lagrange dual function as 
\begin{equation}\label{eq:rat-d}
d(\bw)=\inf_{\begin{subarray}{c} \{\alpha_j\},~\{\beta_j\}\\
            \sum_{j=1}^m w_j |q(x_j)|^2=1\end{subarray}}\sum_{j=1}^m w_j |f_j q(x_j)-p(x_j)|^2,
\end{equation}  
where 
\begin{equation}\label{eq:S}
\bw\in {\cal S}:=\{\bw=[w_1,\dots,w_m]^{\T}\in \bbR^m: \bw\ge 0 ~{\rm and } ~\sum_{j=1}^m w_j=1\}.
\end{equation}
Consequently, the associated dual problem reads as 
\begin{equation}\label{eq:rat-dual}
    \max_{\bw\in {\cal S}}d(\bw).
\end{equation}

To facilitate the computation and simplify the dual function $d(\bw)$, for the numerator $p(x)$ and the denominator $q(x)$, we introduce the basis matrices 
\begin{subequations}\label{eq:Cpq}
\begin{align}\label{eq:Cp}
C_p&=\left[\begin{array}{ccc|ccc}\frac{y_1}{x_1-t_1} & \cdots & \frac{y_\ell}{x_1-t_{\ell}} &\frac{1}{x_1-t_{\ell+1}}  &\dots&\frac{1}{x_1-t_{n+1}} \\  \vdots&\cdots &\vdots & \vdots& \cdots &\vdots \\\frac{y_1}{x_m-t_1} & \cdots & \frac{y_\ell}{x_m-t_{\ell}} &\frac{1}{x_m-t_{\ell+1}}  &\dots&\frac{1}{x_m-t_{n+1}}\end{array}\right]=:[C_{p,1}, C_{p,2}]\in\bbC^{m\times(n+1)},\\\label{eq:Cq}
C_q&=\left[\begin{array}{ccc|ccc}\frac{1}{x_1-t_1} & \cdots & \frac{1}{x_1-t_{\ell}} &\frac{1}{x_1-t_{\ell+1}}  &\dots&\frac{1}{x_1-t_{n+1}} \\  \vdots&\cdots &\vdots & \vdots& \cdots &\vdots \\\frac{1}{x_m-t_1} & \cdots & \frac{1}{x_m-t_{\ell}} &\frac{1}{x_m-t_{\ell+1}}  &\dots&\frac{1}{x_m-t_{n+1}}\end{array}\right]=:[C_{q,1},C_{q,2}]\in\bbC^{m\times(n+1)}.
\end{align} 
\end{subequations} 
Indeed, we have $C_{p,1}=C_{q,1}\cdot \diag(y_1,\dots,y_\ell)$ and $C_{p,2}=C_{q,2}$; also  the coefficient vector $ \left[\begin{array}{c}\ba \\\bb\end{array}\right]\in \bbC^{2(n+1)}$ where 
\begin{equation}\label{eq:coeffab}
\ba=[\alpha_1,\dots,\alpha_{\ell},\alpha_{\ell+1},\dots,\alpha_{n+1}]^{\T} ~\mbox{and}~\bb=[\beta_1,\dots,\beta_{\ell},\beta_{\ell+1},\dots,\beta_{n+1}]^{\T}.
\end{equation}
It should be noticed from the denominator of \eqref{eq:baryform} that 
\begin{equation}\label{eq:c}
\alpha_j=\beta_j,~~\forall 1\le j\le \ell \Longleftrightarrow \left[\begin{array}{cccc}I_\ell &0_{\ell\times (n+1-\ell)}    &-I_{\ell}&0_{\ell\times (n+1-\ell)}  \end{array}\right]  \left[\begin{array}{c}\ba \\\bb\end{array}\right]=\bzs.
\end{equation}
Since  $\left[\begin{array}{c}\ba \\\bb\end{array}\right]$ satisfying   \eqref{eq:c} can be parametrized as 
\begin{equation}\label{eq:c2}
\left[\begin{array}{c}\ba \\\bb\end{array}\right] = \left[\begin{array}{ccc}I_{\ell} &   &   \\  &I_{n+1-\ell} &  \\I_{\ell} &   &   \\  & &I_{n+1-\ell} 
\end{array}\right]\bc=:N\bc,~N\in \bbR^{ 2(n+1)\times (2n+2-\ell) }, ~\bc\in \bbC^{2(n+1)-\ell}, 
\end{equation}
we can rewrite the   dual function $d(\bw)$ of \eqref{eq:rat-d} as 
\begin{equation}\label{eq:dual}
d(\bw)=\min_{\begin{subarray}{c} \bc\in \bbC^{2(n+1)-\ell}\\\left\|\sqrt{W}[0_{m\times (n+1)},C_q]N\bc\right\|_2=1
\end{subarray}}\left\|\sqrt{W}[C_p,-FC_q]N\bc\right\|_2^2,
\end{equation}
where $F=\diag(\Bf)$ and $W=\diag(\bw)$. 
By the formulations of the objective function as well as the constraints    of \eqref{eq:rat-d}, we know that the infimum in \eqref{eq:rat-d} is attainable according to  \cite[Theorem 4.1]{lilb:2013}, and thus we have expressed the computation of the dual function $d(\bw)$ as the minimum of \eqref{eq:dual}.

We put a remark on the matrix $N$ in \eqref{eq:dual}. Indeed, $N$  is a basis of the null space: $${\rm Ker}\left(\left[\begin{array}{cccc}I_\ell &0_{\ell\times (n+1-\ell)}    &-I_{\ell}&0_{\ell\times (n+1-\ell)}  \end{array}\right]\right),$$  
and for any $$A=[A_1,A_2,A_3,A_4]\in\bbC^{m\times 2(n+1)},~A_{1},A_3\in \bbC^{m\times \ell},~A_{2},A_4\in \bbC^{m\times (n+1-\ell)},$$
it follows 
$AN=[A_1+A_3,A_2,A_4]\in \bbC^{m\times (2n+2-\ell)}.$

\section{Computation of the dual function}\label{subsec:dualfun}
Before discussing duality between   \eqref{eq:baryInterpMinmax} and the dual problem \eqref{eq:rat-dual}, we first consider the computation of  the minimization \eqref{eq:dual} associated with the dual function $d(\bw)$.  
The following is the optimality condition for \eqref{eq:dual}.

\begin{lemma}\label{lem:KKT}
Given $\{(x_j,f_j)\}_{j=1}^m~(m\ge 2n+2-\ell)$, for  any $ \bw\in {\cal S}$, let $d(\bw)$ be the minimum of  \eqref{eq:dual} and  $\bc(\bw)$ be the associated solution for the minimization.  Denote $A=\sqrt{W}[C_p,-FC_q]N$ and $B=\sqrt{W}[0_{m\times (n+1)},C_q]N$, where $F=\diag(\Bf)\in \bbC^{m\times m}$, $W=\diag(\bw)\in \bbR^{m\times m}$,   $C_p,C_q$ are defined in \eqref{eq:Cpq} and  $N$ is given in \eqref{eq:c2}. Then  $(d(\bw),\bc(\bw))$ is the eigenpair associated with the smallest eigenvalue of the generalized eigenvalue problem $A^{\HH}A\bc(\bw)=d(\bw)B^{\HH}B\bc(\bw)$ subject to the normalized condition $\|B\bc(\bw)\|_2=1$. Thus, $A^{\HH}A-d(\bw)B^{\HH}B$ is positive semidefinite. 
\end{lemma}
\begin{proof}
This is from \cite[Theorem 8.7.1]{govl:2013} or \cite{lilb:2013}  because \eqref{eq:dual} can be equivalent to minimizing 
the quadratic form $\bc^{\HH} A^{\HH}A\bc$ subject to the quadratic constraint $\bc^{\HH} B^{\HH}B\bc=1$, for which the minimum is attainable due to  \cite[Theorem 4.1]{lilb:2013}.
\end{proof}
 
In the implementation, we provide a computational efficient way to solve the minimization \eqref{eq:dual}.
 Split  the basis matrices in \eqref{eq:Cpq} as
\begin{equation}\label{eq:splitCpq}
C_p=[C_{p,1}, C_{p,2}],~C_q=[C_{q,1}, C_{q,2}], ~~C_{p,1}, C_{q,1}\in \bbC^{m\times \ell}, ~C_{p,2}=C_{q,2}\in \bbC^{m\times (n+1-\ell)}.
\end{equation}
 Let 
 \begin{equation}\label{eq:QRWCq}
 \sqrt{W}C_q=Q_qR_q,~{\rm and}~\sqrt{W}C_{p,2}=Q_{p,2}R_{p,2}
 \end{equation} 
be the thin QR decompositions, where $Q_q\in \bbC^{m\times (n+1)}$, $R_q\in \bbC^{(n+1)\times (n+1)}$, $Q_{p,2}\in \bbC^{m\times (n+1-\ell)}$, $R_{p,2}\in \bbC^{(n+1-\ell)\times (n+1-\ell)}$. Then we have the following proposition.

\begin{proposition}\label{prop:dual_GEP}
{Given $\{(x_j,f_j)\}_{j=1}^m~(m\ge 2n+2-\ell)$, for  any $ \bw\in {\cal S}$ with at least $n+1$ nonzero $w_j$'s, let $d(\bw)$ be the minimum of  \eqref{eq:dual} and  
\begin{equation}\nonumber 
\bc(\bw)=  \left[\begin{array}{c} \bc_1(\bw) \\ \bc_2(\bw)\\ \bc_3(\bw)\end{array}\right] \in \bbC^{2(n+1)-\ell},~ \bc_1(\bw)\in \bbC^{\ell},~ \bc_2(\bw)\in \bbC^{n+1-\ell},~ \bc_3(\bw)\in \bbC^{n+1-\ell},
\end{equation}
be the associated solution. Then, with the notations in \eqref{eq:splitCpq} and \eqref{eq:QRWCq}, $\sqrt{d(\bw)}$ is the smallest singular value of 
\begin{equation}\label{eq:optc1}
(I_{m}-Q_{p,2}Q_{p,2}^{\HH})\sqrt{W}[C_{p,1}-FC_{q,1},-FC_{q,2}]R_q^{-1}
\end{equation}
with the corresponding right singular vector 
$R_q\left[\begin{array}{c}\bc_1(\bw) \\\bc_3(\bw)\end{array}\right]$, and 
\begin{equation}\label{eq:c2w}
\bc_2(\bw)=-R_{p,2}^{-1}Q_{p,2}^{\HH}\sqrt{W}[C_{p,1}-FC_{q,1},-FC_{q,2}]\left[\begin{array}{c}\bc_1(\bw) \\\bc_3(\bw)\end{array}\right].
\end{equation} 
}
\end{proposition}
\begin{proof} Based on Lemma \ref{lem:KKT}, we only need to provide a detailed formulation for this eigenvalue problem $A^{\HH}A\bc(\bw)=d(\bw)B^{\HH}B\bc(\bw)$ with  $\|B\bc(\bw)\|_2=1$.  
To simplify   presentation, we introduce a permutation matrix 
\[
{P = \begin{bmatrix}
I_\ell & \\
&& I_{n+1-\ell} \\
&  I_{n+1-\ell} &
\end{bmatrix} \in \mathbb{R}^{(2n+2-\ell) \times (2n+2-\ell)}.}
\]
and  write  
\begin{align}\nonumber
\sqrt{W}[0_{m\times (n+1)},C_q]N\bc=\sqrt{W}[0,C_q]NPP^{\T}\bc=\sqrt{W}[C_{q}, 0_{m\times (n+1-\ell)}]  \what\bc,
\end{align}
where $\what \bc = P^{\T}\bc=[\bc_1(\bw)^{\T},\bc_3(\bw)^{\T},\bc_2(\bw)^{\T}]^{\T}.$ Similarly, 
\begin{equation}\nonumber
\sqrt{W}[C_p,-FC_q]N\bc=\sqrt{W}[C_{p,1}-FC_{q,1}, -FC_{q,2},C_{p,2}]\what\bc.
\end{equation}
Now, by \cite[Theorem 8.7.1]{govl:2013} or \cite{lilb:2013} again, the solution $\what \bc(\bw)$ can be solved from 
\begin{align*}
\left\{\begin{array}{c}\what A_{1}^{\HH}W\what A_{1}\what \bc_1+\what A_{1}^{\HH}W\what A_{2}\what \bc_2=d(\bw)C_{q}^{\HH}WC_{q}\what \bc_1\\\what A_{2}^{\HH}W\what A_{1}\what \bc_1+\what A_{2}^{\HH}W\what A_{2}\what \bc_2=0,
\end{array}\right.
\end{align*}
where $\what A_1=[C_{p,1}-FC_{q,1}, -FC_{q,2}]\in \bbC^{m\times (n+1)}$, $\what A_2= C_{p,2}\in \bbC^{m\times (n+1-\ell)},$ $\what \bc_1=\left[\begin{array}{c}\bc_1(\bw) \\\bc_3(\bw)\end{array}\right]$ and $\what \bc_2=\bc_2(\bw).$ Using the QR decompositions of $\sqrt{W}C_q=Q_qR_q$ and $\sqrt{W}C_{p,2}=Q_{p,2}R_{p,2}$,  and substituting $\what \bc_2=-(\what A_{2}^{\HH}W\what A_{2})^{-1}\what A_{2}^{\HH}W\what A_{1}\what \bc_1$  into the first equation, we get
\begin{equation}\nonumber
\left(\sqrt{W}\what A_1\right)^{\HH} (I_m-Q_{p,2}{Q_{p,2}}^{\HH})\sqrt{W}\what A_1\what \bc_1=\left(\sqrt{W}\what A_1\right)^{\HH} (I_m-Q_{p,2}{Q_{p,2}}^{\HH})^2\sqrt{W}\what A_1\what \bc_1=d(\bw) R_q^{\HH}R_q\what \bc_1.
\end{equation}
It then can be seen that $\sqrt{d(\bw)}$ is the smallest singular value of  \eqref{eq:optc1} and $R_q \what \bc_1$ is the associated right singular vector. The result \eqref{eq:c2w} also follows. 
\end{proof}

\begin{remark}\label{rkm:twocases}
Two special cases are worth mentioning. 
\begin{itemize}
\item[1)] The traditional discrete rational minimax approximation corresponds to $\ell=0$. In this case, $\bc_2(\bw)$ is null and  
$
\bc(\bw)=  \left[\begin{array}{c} \bc_1(\bw) \\ \bc_3(\bw)\end{array}\right] \in \bbC^{2(n+1)},~ \bc_1(\bw)\in \bbC^{n+1},~ \bc_3(\bw)\in \bbC^{n+1};
$
$\sqrt{d(\bw)}$ is the smallest singular value of 
\begin{equation}\nonumber
(I_{m}-Q_{p}Q_{p}^{\HH})F\sqrt{W}C_{q}R_q^{-1}=(I_{m}-Q_{p}Q_{p}^{\HH})FQ_{q},
\end{equation}
with the right singular vector $R_q\bc(\bw)$, where $\sqrt{W}C_q=Q_{q}R_q$ is the thin QR decomposition. 
\item[2)] For the barycentric formula  $\xi(x)= {\sum_{j=1}^{n+1}\frac{\beta_j y_j}{x-  t_j}}\bigg/{\sum_{j=1}^{n+1}\frac{\beta_j}{x-t_j}}$ associated with the traditional linear rational interpolation \cite{bemi:1997b,bekl:2014,scwe:1986}, i.e., $\ell=n+1$, we know $Q_{p,2}$,  
$\bc_2(\bw)$ and $\bc_3(\bw)$ are null,  and  $\bc(\bw)=\bc_1(\bw)$. In this case, $R_q\bc(\bw)$ is the right singular vector of 
$\sqrt{W}(C_{p}-FC_{q})R_q^{-1}$
associated with its smallest singular value $\sqrt{d(\bw)}$. 
\end{itemize}
\end{remark}
\section{Duality}\label{subsec:duality}
We now develop the dual theory of the minimax approximation \eqref{eq:baryInterpMinmax}. We remark that, relying on the barycentric parameterization of \eqref{eq:baryform}, the derivations of relevant properties of the dual problem \eqref{eq:rat-dual} are essentially similar to that for the classical polynomial \cite{yazz:2023} and rational \cite{zhyy:2025} minimax approximations. 

\subsection{Weak duality}\label{subsubsec:weakdual}

\begin{theorem}\label{thm:q-dual}
Let $\{(t_j,y_j)\}_{j=1}^\ell ~(0\le \ell\le n+1)$ be the interpolation data, and choose ${\cal \what I}=\{t_j\}_{j=\ell+1}^{n+1}\subset\bbC$ so that the support points in ${\cal T}=\{t_j\}_{j=1}^{n+1}$ are distinct for $\xi(x)$ defined in \eqref{eq:baryform}. Given  $\{(x_j,f_j)\}_{j=1}^m~(m\ge 2n+2-\ell)$ with ${\cal T}\cup \{x_j\}_{j=1}^m=\emptyset$, suppose $\xi^*=p^*/q^*\in \scrR_{(n)}$ with $q^*(x_j)\ne 0~(1\le j\le m)$ is a solution to \eqref{eq:baryInterpMinmax} with 
$
\eta_\infty=\|\Bf-\xi^*(\bx)\|_\infty.
$
Then we have the weak duality:
\begin{equation} \label{eq:weakduality}
\max_{\bw\in {\cal S}}d(\bw)\le (\eta_\infty)^2, 
\end{equation}
where ${\cal S}$  and   the dual function $d(\bw)$ are given in  \eqref{eq:S} and \eqref{eq:rat-d} (or equivalently, \eqref{eq:dual}), respectively.
\end{theorem}
\begin{proof}
We  prove that 
$d(\bw)\le (\eta_\infty)^2,  \forall \bw\in {\cal S},$
and therefore, $\max_{\bw\in {\cal S}}d(\bw) \le (\eta_\infty)^2$. 

First, by Theorem \ref{lem:baryrational2}, we know that both the best approximation $\xi^*=p^*/q^*\in \scrR_{(n)}$  of \eqref{eq:bestf}, and the corresponding   $\wtd \xi=\wtd p/\wtd q$ in barycentric formula that achieves the minimum $d(\bw)$ can be parameterized by \eqref{eq:baryform}. 
As  $q^*(x_j)\ne 0$, we can always choose a scaling $\tau$ so that $p_\tau^*= \tau p^*, ~q_\tau^*= \tau q^* $ satisfying $\sum_{j=1}^m  w_j|q_\tau^*(x_j)|^2 =1$. Thus
\begin{align}\nonumber
d(\bw)&=\sum_{j=1}^m w_j |f_j \wtd q(x_j)-\wtd  p(x_j)|^2 \le\sum_{j=1}^m w_j |f_j  q_\tau^*(x_j)-  p_\tau^*(x_j)|^2
 \\\nonumber
&= \sum_{j=1}^m w_j| q_\tau^*(x_j)|^2\cdot \left|f_j-\frac{ p_\tau^*(x_j)}{ q_\tau^*(x_j)}\right|^2 
 \le \sum_{j=1}^m w_j| q_\tau^*(x_j)|^2\cdot \|\Bf-\xi^*\|_\infty^2= (\eta_\infty)^2
\end{align}
leading to the weak duality \eqref{eq:rat-dual}.
\end{proof}

\subsection{Strong duality  and complementary slackness}\label{subsubsec:strongdual}

\begin{theorem}\label{thm:strongdualityRuttan}
Let $\bw^*\in{\cal S}$ be the maximizer of the dual  {problem} \eqref{eq:rat-dual},  and  $(\wtd p,\wtd  q)$ be the associated pair represented in barycentric formula   \eqref{eq:baryform} with the corresponding  $\wtd\beta_j\ne 0 ~(1\le j\le \ell)$ that achieves the minimum  ${d}({\bw}^*)$ of \eqref{eq:rat-d}. Assume  $\wtd q(x_j)\ne 0~(1\le j\le m)$  and 
\begin{equation}\label{eq:strongdualityRuttan}
\sqrt{d(\bw^*)}= \max_{1\le j\le m}\left|f_j-\frac{\wtd p(x_j)}{\wtd q(x_j)}\right|:=e(\wtd \xi).
\end{equation}
Then strong duality
\begin{equation}\label{eq:strongdual}
 \max_{\bw\in {\cal S}} d(\bw)=(\eta_\infty)^2
\end{equation} 
 holds and $\wtd \xi=\frac{\wtd p}{\wtd q}$ is the global solution to \eqref{eq:baryInterpMinmax}; furthermore, 
 \begin{equation}\label{eq:complement}
{\rm (complementary ~slackness)}  ~w_j^*\left(\eta_\infty -|f_j-\wtd \xi(x_j)|\right)=0, \quad\forall j=1,2,\dots,m.
\end{equation} 
\end{theorem}
\begin{proof}
 We prove   strong duality \eqref{eq:strongdual} and $\wtd \xi=\frac{\wtd p}{\wtd q}$ is the global solution to \eqref{eq:baryInterpMinmax} by contradiction. Let $\eta_{\infty}$ be the infimum of \eqref{eq:baryInterpMinmax}. Suppose $\wtd \xi$ is not a global solution to \eqref{eq:baryInterpMinmax}. Then either the infimum of \eqref{eq:baryInterpMinmax} is unattainable and $e(\wtd \xi)>\eta_\infty$, or  is attainable but $\wtd \xi$ is not a solution. In either case, it must hold that $e(\wtd \xi)>\eta_{\infty}$ and we can find\footnote{When the infimum of \eqref{eq:baryInterpMinmax} is unattainable and $e(\wtd \xi)>\eta_\infty$, then by the definition of infimum \eqref{eq:baryInterpMinmax},   there is a $\what \xi$ in the barycentric form \eqref{eq:baryform} with the corresponding $\what \beta_j\ne 0~(1\le j\le \ell)$ satisfying  $\eta_{\infty}<e(\what \xi)<e(\wtd \xi)$; when  the infimum of \eqref{eq:baryInterpMinmax} is attainable but $\wtd \xi$ is not a solution, we simply take $\what \xi$ as the rational minimax approximant of \eqref{eq:baryInterpMinmax}. } a $\what \xi\in \scrR_{(n)}$ with $\what \xi=\what p/\what q$  in the barycentric formula \eqref{eq:baryform} satisfying $\what q(x_j)\ne 0 ~(1\le j\le m)$ with the corresponding $\what \beta_j\ne 0~(1\le j\le \ell)$  and
 \begin{equation}\label{eq:etawtd}
 e(\what  \xi):=\max_{1\le j\le m}|f_j-\what \xi(x_{j})|<e(\wtd \xi).
 \end{equation}
 {Let $\what \ba, \what \bb\in \bbC^{n+1}$ be in the form of \eqref{eq:coeffab}  associated coefficients vectors of $\what p$ and $\what q$ represented in the barycentric form \eqref{eq:baryform}. Let $\what \bc\in \bbC^{2n+2-\ell}$ be the corresponding vector satisfying $\left[\begin{array}{c}\what\ba \\\what\bb\end{array}\right]  =N\what\bc$ given in \eqref{eq:c2}.} By Lemma  \ref{lem:KKT}, the matrix $H_{\bw^*}=A^{\HH}A-d(\bw^*)B^{\HH}B$ is positive semidefinite. Thus 
 \begin{align}\nonumber
 0\le&~{ \what \bc^{\HH}\left(A^{\HH}A-d(\bw^*)B^{\HH}B\right) \what \bc }\\\nonumber
 =&\sum_{j=1}^m w_j^* \left( |f_{j}(x_j)\what q(x_j)-\what  p (x_j)|^2 -d(\bw^*)|\what  q(x_j)|^2\right)\\\nonumber
=&\sum_{j=1}^mw_j^*| \what  q(x_j)|^2\left( |f_{j}(x_j)-\what \xi(x_j)|^2-e(\wtd \xi)^2\right)\quad \quad {\rm (by ~\eqref{eq:strongdualityRuttan})}\\\nonumber
\le &\sum_{j=1}^mw_j^*| \what  q(x_j)|^2\left(\|\Bf-\what \xi(\bx)\|_{\infty}^2-e(\wtd \xi)^2\right).  
 \end{align}
 As $\bw\ne \bzs$ and $\what  q(x_j)\ne 0~(1\le j\le m)$, this implies   $e(\what \xi)=\|\Bf-\what \xi(\bx)\|_{\infty}\ge e(\wtd \xi)$, a contradiction against \eqref{eq:etawtd}. This proves that the infimum in  \eqref{eq:baryInterpMinmax} is attainable by $\wtd \xi(x)\in \scrR_{(n)}.$ Thereby, we have $\eta_{\infty}=e(\wtd \xi)$. Since $\eta_{\infty}$ is an upper  bound of the dual problem by weak duality \eqref{eq:weakduality},  $\eta_{\infty}=e(\wtd \xi)=\sqrt{d(\bw^*)}$ ensures strong duality \eqref{eq:strongdual}. 

To prove  the complementary slackness \eqref{eq:complement}, since   strong duality \eqref{eq:strongdual} holds and  $(\wtd p,\wtd q)$ is feasible (by a scaling for $\wtd q$ using $\bw^*$) for the minimization \eqref{eq:rat-d} at $\bw=\bw^*$, we have 
\begin{align*}
(\eta_\infty)^2=d_2(\bw^*)& = \sum_{j=1}^m w^*_j|f(x_j)\wtd q(x_j)-\wtd p(x_j)|^2 \\
&=  \sum_{j=1}^m w^*_j|\wtd q(x_j)|^2 |f(x_j)-\wtd \xi(x_j)|^2\\
&\le  (\eta_\infty)^2\sum_{j=1}^m w^*_j|\wtd q(x_j)|^2=(\eta_\infty)^2,
\end{align*}
which yields $$w^*_j|\wtd q(x_j)|^2 \left((\eta_\infty)^2-|f(x_j)-\wtd \xi(x_j)|^2\right)=0, ~\forall 1\le j\le m.$$ As $\wtd q(x_j)\ne 0 ~(1\le j\le m)$ by assumption, it follows $w_j^*(\eta_\infty -|f(x_j)-\wtd \xi(x_j)|)=0$ $(1\le j\le m).$\end{proof}

We should remind the reader that Theorem \ref{thm:strongdualityRuttan} does not claim that the minimax approximant\footnote{The problem \eqref{eq:baryInterpMinmax} may not have a solution, i.e., the infimum is unachievable.}, if exists,  of \eqref{eq:baryInterpMinmax} can always be obtained {\it via} solving the dual problem \eqref{eq:rat-dual}; it is only a sufficient condition. To see this, suppose   \eqref{eq:baryInterpMinmax} is attainable with the infimum $\eta_\infty>0$. Let $\bw^*$ be the maximizer of the dual problem \eqref{eq:rat-dual}. In some cases (for example $\Bf=\bzs$), the matrix in \eqref{eq:optc1} can be column rank-deficient and thus $d(\bw^*)=0$. This invalidates the sufficient condition \eqref{eq:strongdualityRuttan}, and therefore, the minimax approximant of \eqref{eq:baryInterpMinmax} cannot be obtained from the dual problem.

\subsection{The set of extreme points}\label{subsec:extremep}
In the classical real polynomial minimax approximation, the so-called {\it equioscillation property} \cite[Theorem 24.1]{tref:2019a} serves a necessary and sufficient condition for the minimax approximant; it also appears very useful for designing numerical methods, including the famous Remez algorithm \cite{fint:2018,patr:2009}. The equioscillation property can be extended to  the complex polynomial minimax case \cite{rish:1961} as well as to the complex rational minimax approximations \cite{will:1972}. One of important notions there is the set of {\it extreme points} (also known as   {\it reference points}). Specifically, for a given approximant $\xi\in \scrR_{(n)}$ for $\Bf$, we define the set of extreme points as 
\begin{equation}\label{eq:extremalset}
 {\cal  X}_e(\xi):=\left\{x_j\in {\cal X}:\left|f_j-\xi(x_j)\right|=\max_{x_j\in {\cal X}}\left|f_j-\xi(x_j)\right|\right\}. 
\end{equation}
In the case of traditional rational minimax approximation \eqref{eq:bestf} (i.e., $\ell=0$)  in $\scrR_{(n)}$, a result \cite[Theorem 2.5]{gutk:1983} implies that for an irreducible rational minimax solution $\xi^*=p^*/q^*~(p^*,q^*\in \bbP_n)$ of \eqref{eq:bestf}, ${\cal  X}_e(\xi^*)$ contains at least $2n+2-\upsilon(p^*,q^*)$ points, where $\upsilon(p^*,q^*)=\min(n-\deg(p^*),n-\deg(q^*))$ is the {\it defect} of $\xi^*$. For our interpolation constrained rational minimax approximation \eqref{eq:rat-d}, we next provide a lower bound for the cardinality $|{\cal X}_e(\wtd \xi)|$ of ${\cal X}_e(\wtd \xi)$.

\begin{theorem}\label{thm:extremep}
Under the assumptions of Theorem \ref{thm:strongdualityRuttan}, for $\wtd \xi=\wtd p/\wtd q$,  
the set of extreme points ${\cal X}_e(\wtd \xi)$ defined as  \eqref{eq:extremalset} contains at least $n+2-\ell$ nodes.
\end{theorem}
\begin{proof}
Under these assumptions,  Theorem \ref{thm:strongdualityRuttan} implies that $\wtd \xi=\wtd p/\wtd q$ is a minimax approximant for \eqref{eq:rat-d} and the complementary slackness \eqref{eq:complement} holds. For the trivial case with $\eta_\infty=0$, we know that $\wtd \xi$ is a rational interpolant for the vector $\Bf$, and the conclusion is true trivially. In the following, we consider $\eta_\infty>0$. 

Note that in \eqref{eq:complement}, the condition $w_j^*>0$ implies that $x_j$ is an extreme point due to $\eta_\infty -|f_j-\wtd \xi(x_j)|=0$ and $\eta_\infty=e(\wtd \xi)$. Hence, we next prove the result by contradiction.  Suppose without loss of generality that  $w_j^*>0$ for $j=1,\dots,k$ ($k\le n+1-\ell$) and $w_j^*=0$ for $k+1\le j\le m$. The idea is to construct a new pair $(\what p, \wtd q)$ represented in barycentric formula   \eqref{eq:baryform}   satisfying
\begin{equation}\label{eq:newp}
\what p(x_j)=f_j\wtd q(x_j), ~1\le j\le k.
\end{equation}
Thus $(\what p, \wtd q)$ is feasible for \eqref{eq:rat-dual} but the corresponding objective value is
\begin{equation}\nonumber
\sum_{j=1}^mw_j^*\left|\what p(x_j)-f_j\wtd q(x_j)\right|^2=\sum_{j=1}^kw_j^*\left|\what p(x_j)-f_j\wtd q(x_j)\right|^2=0.
\end{equation}
Because $d(\bw^*)$ is the minimum of \eqref{eq:rat-dual}, it leads to $d(\bw^*)=0$, a contradiction against $d(\bw^*)=\eta_\infty^2>0$. 

To see \eqref{eq:newp}, let 
\begin{equation}\nonumber
\wtd p(x)=\sum_{j=1}^\ell\frac{\wtd \beta_j y_j}{x-t_j}+\sum_{j=\ell+1}^{n+1}\frac{\wtd \alpha_j}{x-t_j},~~\wtd q(x)=\sum_{j=1}^\ell\frac{\wtd \beta_j}{x-t_j}+\sum_{j=\ell+1}^{n+1}\frac{\wtd\beta_j}{x-t_j}.
\end{equation}
As $\{\frac{1}{x-t_j}\}_{j=\ell+1}^{n+1}$ are $n+1-\ell$ linearly independent functions, we can choose the parameters $\{\what \alpha_j\}_{j=\ell+1}^{n+1}$ to meet the $k$ (note $k\le n+1-\ell$) conditions \eqref{eq:newp} and construct 
$$
\what p(x)=\sum_{j=1}^\ell\frac{\wtd \beta_j y_j}{x-t_j}+\sum_{j=\ell+1}^{n+1}\frac{\what \alpha_j}{x-t_j}.
$$
Due to the barycentric form, the resulting $\what \xi=\what p/\wtd q$ satisfies $\what \xi(t_j)=y_j~(1\le j\le \ell)$ and the proof is complete.
\end{proof}

We remark that for $\ell=0$, the lower bound $n+2$ does not involve the defect\footnote{For the pair $(\wtd p,\wtd  q)$  represented in barycentric formula   \eqref{eq:baryform}, we can similarly define the defect as  $\upsilon(\lambda(x)\wtd p,\lambda(x)\wtd q)$ where $\lambda(x)=\prod_{j=1}^{n+1}(x-t_j)$.} $\upsilon(\wtd p,\wtd q)$ of $\wtd \xi$ and  is a less sharp bound than $2n+2-\upsilon(\wtd p,\wtd q)$ \cite[Theorem 2.5]{gutk:1983}. {A recent work \cite{zhan:2026} discusses the optimality conditions for the rational minimax approximation and bridges Ruttan's conditions \cite{rutt:1985,this:1993}  to the dual-based framework in \eqref{eq:rat-dual}.}  For the real and continuous version of \eqref{eq:baryInterpMinmax0} with $\ell>0$ where ${\cal X}=[a,b]$, \cite[Theorem 4]{lixu:1990}  shows that the lower bound of the cardinality $|{\cal X}_e(\wtd \xi)|$ of ${\cal X}_e(\wtd \xi)$ is $2n+2-\ell-\upsilon(\wtd p,\wtd q)$.   
Technically, establishing the bound $2n+2-\upsilon(\wtd p,\wtd q)$ for the complex case requires other characterizations (see e.g., \cite[Section 3]{lixu:1990} for the real case) for $\wtd \xi$, and our treatment in Theorem \ref{thm:extremep} based on strong duality seems not sufficient to obtain that bound. Nevertheless,  in our numerical results,  we observed that $|{\cal X}_e(\wtd \xi)|\ge 2n+2-\upsilon(\wtd p,\wtd q)$ for $\ell=0$.
\section{The {\sf b-d-Lawson} iteration}\label{sec:bdlawson}
In the recent work \cite{yazz:2023}  for the polynomial minimax approximation, it has been revealed that the traditional Lawson's iteration \cite{laws:1961} is an ascent iteration for solving the dual problem. Within the same dual framework, \cite{zhyy:2025} developed the dual problem and proposed a dual-based Lawson iteration, {\sf d-Lawson}, for the traditional rational minimax approximation \eqref{eq:bestf} (i.e., $\ell=0$). Convergence of {\sf d-Lawson} has also been discussed in \cite{zhha:2025}. Due to the barycentric representation in \eqref{eq:baryform}, the interpolation constrained rational minimax approximation \eqref{eq:baryInterpMinmax} admits the same structure as the traditional rational minimax approximation \eqref{eq:bestf}, and thus, {\sf d-Lawson} can be extended in parallel. This is our proposed method {\sf b-d-Lawson} presented in Algorithm \ref{alg:bdLawson}.  
 
 A recommended initial guess of $\bw$ is $\bw^{(0)}=[1,\dots,1]^{\T}/m\in \bbR^{m}$. Other than the given support points $\{t_j\}_{j=1}^\ell$, the remaining $\{t_j\}_{j=\ell+1}^{n+1}$ can be chosen in various ways. In principle, any distinct $\{t_j\}_{j=\ell+1}^{n+1}$ that are different from $x_j$ for $1\le j\le m$ and $t_j$ for $1\le j\le \ell$ can be used in the barycentric formula \eqref{eq:baryform}. Nevertheless, a proper idea is to invoke the adaptive way in the AAA method \cite{nase:2018}:   choose the first  $\{x_{I_j}\}_{j=1}^{n+1-\ell} \subseteq \{x_j\}_{j=1}^m$ support points as in AAA   \cite{nase:2018},  and then perturb them slightly to form  $\{t_j\}_{j=\ell+1}^{n+1}$. In our numerical testing, we set $t_{\ell+j}=x_{I_j}+\frac{1}{10 m}~(1\le  j\le n+1-\ell).$

The convergence of Lawson's iteration for the polynomial and rational minimax approximations is highly nontrivial. Particularly, for the polynomial case, convergence analysis of Lawson's iteration \cite{laws:1961} has been discussed (e.g., \cite{clin:1972,yazz:2023,zhha:2025}). However, for the rational case, convergence of Lawson's iteration is not well understood. A recent work \cite{zhha:2025} provides the convergence analysis for {\sf d-Lawson}. Based on the barycentric representation, the proposed  {\sf b-d-Lawson} with $\ell=0$ is a new implementation of {\sf d-Lawson} to compute type $(n,n)$ rational minimax approximants. Therefore, the convergence analysis \cite{zhha:2025} can be helpful in understanding the behavior of {\sf b-d-Lawson} in this situation. Specifically for $\ell=0$, we point out that when $d(\bw^{(k)})$ is a simple eigenvalue of the matrix pencil $(A^{\HH}A,B^{\HH}B)$ in Lemma \ref{lem:KKT}, then there is a $\rho_0>0$ so that for any $\rho\in (0,\rho_0)$, it holds that $d(\bw^{(k+1)})\ge d(\bw^{(k)})$;  
on the other hand, for a general case where $0 < \ell \leq n+1$, the convergence analysis of {\sf b-d-Lawson}  requires additional considerations and lies beyond the scope of this paper. 

\begin{algorithm}[h!!!]
\caption{A barycentric-dual-Lawson iteration ({\sf b-d-Lawson}) for \eqref{eq:baryInterpMinmax}} \label{alg:bdLawson}
\begin{algorithmic}[1]
\renewcommand{\algorithmicrequire}{\textbf{Input:}}
\renewcommand{\algorithmicensure}{\textbf{Output:}}
\REQUIRE Interpolation data $\{(t_j,y_j)\}_{j=1}^\ell$, sample data  $\{(x_j,f_j)\}_{j=1}^m ~(m\ge 2n+2-\ell)$,  a relative tolerance for strong duality $\epsilon_r>0$ and the maximum number $k_{\rm maxit}$ of iterations.
\ENSURE  The minimax rational approximation $\xi^*$ of \eqref{eq:baryInterpMinmax} represented in the barycentric formula \eqref{eq:optimalpq}.
        \smallskip

\STATE  (Support points) Choose $\{t_j\}_{j=\ell+1}^{n+1}$ to form the $n+1$ support points $\{t_j\}_{j=1}^{n+1}$; 

\STATE (Initialization) Let $k=0$ and choose $0<\bw^{(0)}\in {\cal S}$; 

\STATE (Compute the dual function) Compute  $d(\bw^{(k)})$ and the associated vector $\xi^{(k)}(\bx)=\bp^{(k)}./\bq^{(k)}$   according to  Proposition \ref{prop:dual_GEP};

\STATE (Stop rule and output)  Stop and return  $\xi^*=\xi^{(k)}$ either if $k\ge k_{\rm maxit}$ or 
\begin{equation}\nonumber
\epsilon(\bw^{(k)}):=\left|\frac{\sqrt{d(\bw^{(k)})}-e(\xi^{{(k)}})}{e(\xi^{{(k)}})}\right|<\epsilon_r,~~{\rm where}~~e(\xi^{{(k)}})=\|\Bf-\xi^{(k)}(\bx)\|_\infty;
\end{equation}

\STATE (Update weights) Update the weight vector $\bw^{(k+1)}$ according to   ($0<\rho\le 1$)
\begin{equation}\nonumber
w_j^{(k+1)}=\frac{w_j^{(k)}\left|f_j-\xi^{(k)}(x_j)\right|^{\rho}}{\sum_{i}w_i^{(k)}\left|f_i-\xi^{(k)}(x_j)\right|^{\rho}},~~\forall j,
\end{equation} 
and   goto Step 3 with $k=k+1$.
\end{algorithmic}
\end{algorithm}  

\section{Numerical experiments}\label{sec:numerical}

To report our numerical experiments on {\sf b-d-Lawson}\footnote{The MATLAB code of {\sf b-d-Lawson} is available at\\  \url{https://ww2.mathworks.cn/matlabcentral/fileexchange/180151-the-b-d-lawson-method}.}, we implement it    
 in {MATLAB R2018a and carry out numerical testing on a 13-inch {MacBook Air} with {an} M2 chip and 8Gb memory.} The unit machine roundoff {is} $\tu=2^{-52}\approx 2.2\times 10^{-16}$.  Fixing  a  maximal number $k_{\rm max}$ of Lawson's iteration and the Lawson exponent $\rho=1$, we denote it by {\sf b-d-Lawson}($k_{\rm max}$) with initial $\bw^{(0)}=\be/m$. Similarly, {{\sf d-Lawson}($k_{\rm maxit}$)} and AAA($k_{\rm maxit}$) represent   {\sf d-Lawson}\footnote{The MATLAB code of {\sf d-Lawson} is available at\\ \url{https://ww2.mathworks.cn/matlabcentral/fileexchange/167176-d-lawson-method}.} \cite{zhyy:2025} and AAA-Lawson\footnote{
In the package of  {\tt Chebfun}, AAA($k_{\rm maxit}$) can be called via {\tt aaa(f,x, 'lawson',k)}. In particular, AAA(0) represents the basic AAA algorithm \cite{nase:2018} without Lawson's iteration.} \cite{nase:2018} with the maximal   $k_{\rm max}$ Lawson's iterations, respectively.

\begin{example}\label{eg1:l=0}
 In our first illustration, we apply  {\sf b-d-Lawson}  to the basic function 
$f(x) = |x|,  x\in[ -1,1]$
 without interpolation condition, i.e., $\ell=0$.  Choosing equally spaced nodes 
$x_{j} = -1 + \frac{2j}{m-1},\; 0\leq j\leq m-1$ with $m = 20000$ and corresponding $f_j=|x_j|$, we list the maximum errors from  {\sf b-d-Lawson}($40$), {\sf d-Lawson}($40$) and AAA($40$) in Table \ref{table7_1} for various types $(n,n)$ of the rational approximants. Note that, in this case, the difference between {\sf b-d-Lawson}($40$) and {\sf d-Lawson}($40$) mainly lies in their representations for the rational functions:  {\sf b-d-Lawson} uses the barycentric formula while {\sf d-Lawson} relies on the monomial basis equipped with  the Vandermonde with Arnoldi process (V+A). The slight improvements for some types illustrate that {\sf b-d-Lawson}  is able to produce competitive  rational approximant as {\sf d-Lawson}($40$) and AAA($40$). This example demonstrates that   {\sf b-d-Lawson} provides another effective approach for the standard discrete rational minimax approximation.

\begin{table}[h!!!]
\caption{\small Maximum errors for approximating $|x|$ in $[-1,1]$}
\hskip -20mm
\begin{center}
\begin{tabular}{|c|c|c|c|}
\hline 
$(n,n)$ &     $ e_{\tt {b-d-Lawson}(40)}(\xi) $   &      
 $ e_{\tt d-Lawson(40)}(\xi) $    &    $ e_{\tt AAA-Lawson(40)}(\xi) $  \\ 
  \hline 
(4,4) &      8.5506e-03 &      8.5949e-03 &      8.5438e-03  \\ \hline
(8,8) &      7.4051e-04 &      7.4623e-04 &      7.3908e-04  \\ \hline
(12,12) &      1.3342e-04 &      1.1308e-04 &      5.4586e-04  \\ \hline
(16,16) &      1.7130e-05 &      1.8650e-05 &      2.0991e-05  \\ \hline
(20,20) &      5.8606e-06 &      3.0925e-06 &      7.4341e-06  \\ \hline
(24,24) &      3.9164e-07 &      4.4134e-07 &      4.0632e-07  \\ \hline
(28,28) &      5.1226e-08 &      6.3754e-08 &      5.2085e-08  \\ \hline
(32,32) &      6.2480e-09 &      6.9913e-09 &      7.0380e-09  \\ \hline
(36,36) &      7.3968e-10 &      8.3275e-10 &      8.2654e-10  \\ \hline
(40,40) &      1.0765e-10 &      9.2506e-11 &      1.6092e-10  \\ \hline
\end{tabular}
\end{center}
\label{table7_1}
\end{table}%
\end{example}

\begin{example} \label{eg2:lne0}
Now, we include interpolation conditions and consider \eqref{eq:baryInterpMinmax} with $\ell>0$.
For the real function 
\begin{equation}\label{eq:f2}
f(x) = \frac{1}{\sqrt{ 1 + 100(x-0.5)^2 } }  + \frac{1}{  1 + 100(x+0.5)^2 }, \quad -1\leq x\leq 1,
\end{equation}
we choose equally spaced nodes $x_{j} = -1 + \frac{2j}{m-1},\; 0\leq j\leq m-1$ with $m = 20000;$ also, we set three interpolation conditions $\{(t_j,y_j)\}_{j=1}^3$ for {\sf b-d-Lawson}($40$), where $t_1=-1, t_2=0,t_3=1$ are in the interval $[-1,1]$ and $y_j=f(t_j)$ for $j=1,2,3$. For {\sf d-Lawson}($40$) and AAA($40$), the three interpolation conditions together with $\{x_j,f(x_j)\}_{j=0}^{m-1}$ form the sample data.  

The results of the three methods for the approximation $\xi\in \scrR_{(6)}$ are plotted in Figure \ref{figure7_1}. It is clear that only  {\sf b-d-Lawson}($40$) fulfills the three interpolation conditions; also, due to these interpolation conditions in {\sf b-d-Lawson}($40$), the maximum error of  {\sf d-Lawson}($40$) and AAA(40) are smaller than that of {\sf b-d-Lawson}($40$). Moreover, it can be seen that there are $11=2 n +2-\ell$ extreme points in the error curve of {\sf b-d-Lawson}($40$), while there are $14= 2n +2$ extreme points for  {\sf d-Lawson}($40$) and AAA(40). 

To demonstrate the monotonic convergence of {\sf b-d-Lawson} and {\sf d-Lawson} with respect to the dual function values and  strong duality, we plot sequences $\{d(\bw^{(k)})\}$ and maximum errors  $\{e(\xi^{(k)})\}$  versus the iteration $k$   in the bottom two subfigures in Figure \ref{figure7_1}. Note that as $k$ increases, the duality gap $e(\xi^{(k)})-\sqrt{d(\bw^{(k)})}$ gets smaller, implying the sufficient condition \eqref{eq:strongdual} is fulfilled.

\begin{figure}[h!!!]
\begin{center}
{\includegraphics[width= 5in ]{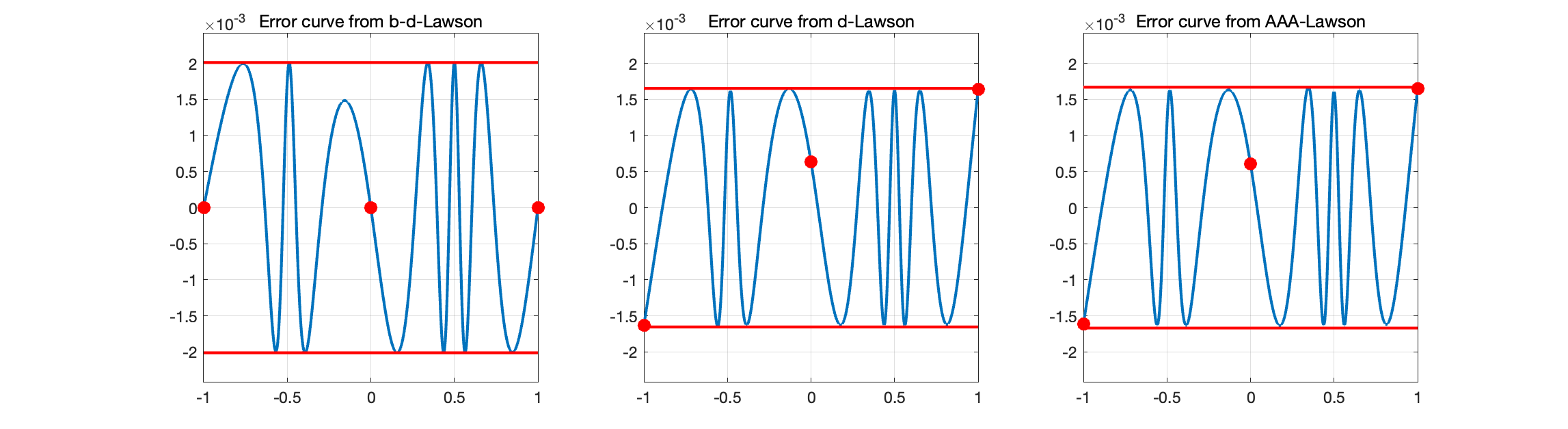} 
 \includegraphics[width= 5in, height=1.25in]{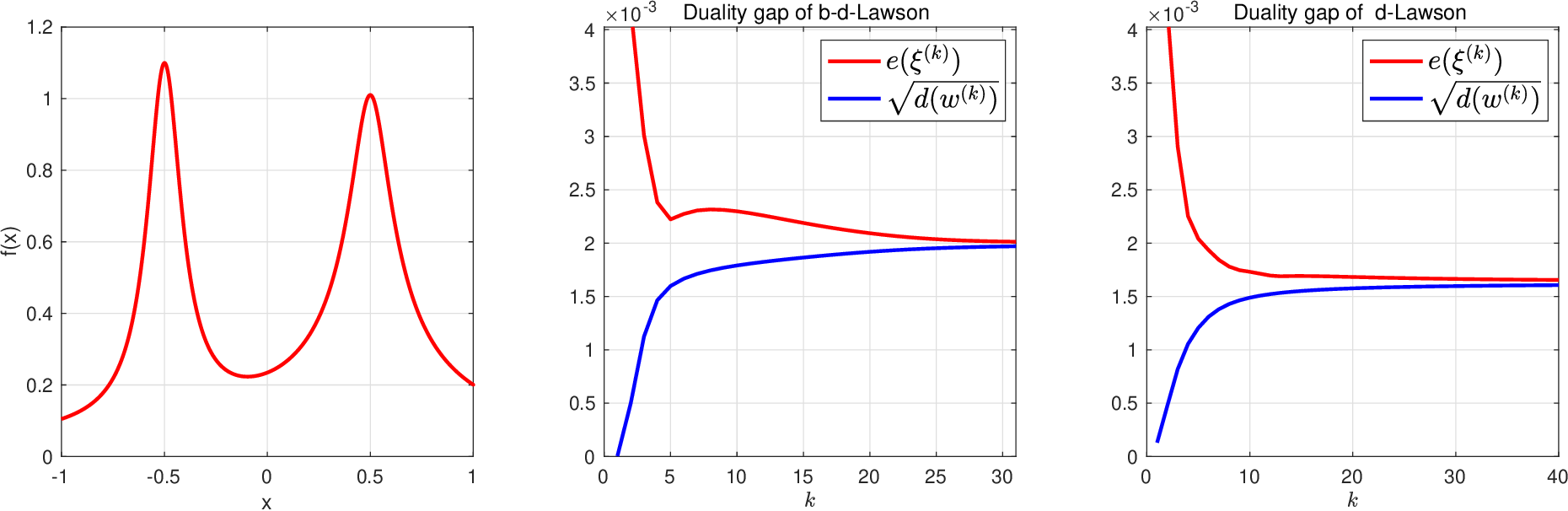}}
\caption{\small The top row: error curves from {\sf b-d-Lawson}($40$), {\sf d-Lawson}($40$) and AAA($40$) of the approximants of type  $(6,6)$, respectively.  Note that there are $11= 2n +2-\ell$ extreme points in the error curve from {\sf b-d-Lawson}($40$), while  $14= 2n +2$ extreme points for  {\sf d-Lawson}($40$) and AAA(40). The bottom row: (bottom-left) the function $f(x)$ in \eqref{eq:f2}; (bottom-middle) the dual objective function values and maximum errors versus the iteration $k$ of {\sf b-d-Lawson}; (bottom-right) the dual objective function values and maximum errors versus the iteration $k$ of {\sf d-Lawson}. }
\label{figure7_1}
\end{center}
\end{figure}
\end{example}

\begin{example} \label{eg3:lne0}
We approximate the Riemann zeta function $f(z)=\zeta(z)$ by a rational function 
using the data in the critical line $L: z = 0.5 + {\tt i}  t, t\in [-50, 50]$ where  ${\tt i}=\sqrt{-1}$. In this testing, we use  $m = 200 $ equally spaced points in $L$ as sampled data $\{x_j,f_j\}_{j=1}^m$ with $f_j=\zeta(x_j)$  evaluated by MATLAB zeta function, and impose the first 11 positive non-trivial roots $\{t_j\}_{j=1}^{11}$ of $\zeta(z)$ as the interpolation conditions, i.e., $\ell=11$.  The image of the computed approximant $\xi(z)$ from {\sf b-d-Lawson}($40$) and errors of the sampled nodes in  $L$ are plotted  in the left two subfigures in Figure \ref{figure7_2}, respectively; moreover, the phase portrait  of $\xi(z)\approx \zeta(z)$ in a striped region is presented in the right subfigure.
\begin{figure}[h!!!]
\hskip -12mm
{\includegraphics[width= 7in ]{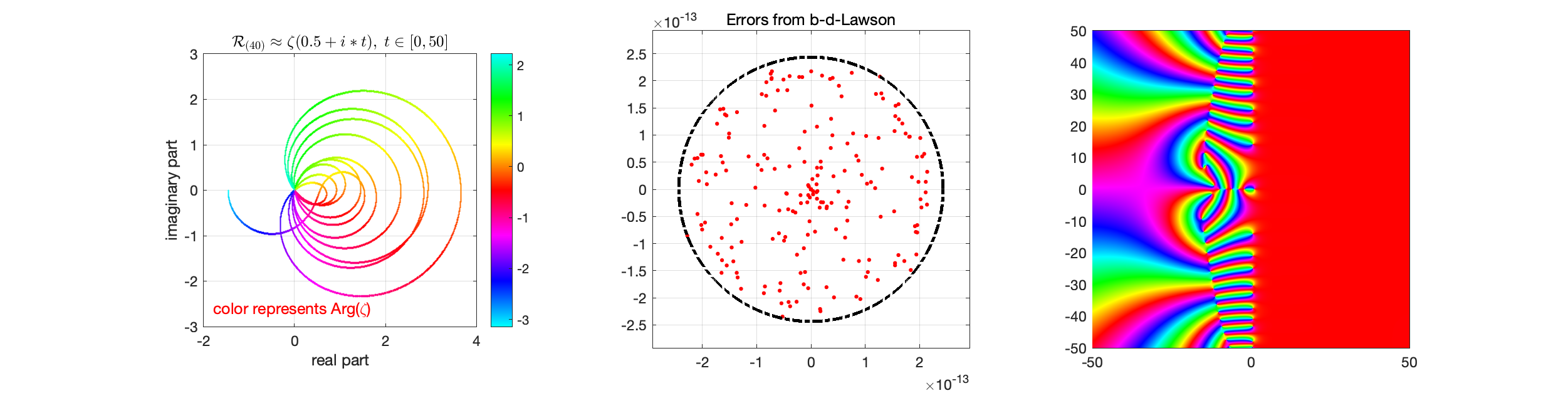} }
\caption{Approximation of the Riemann zeta function $\zeta(z)$:  (left)  the image of the computed approximant $\xi\in \scrR_{(40)}$ of the sampled nodes in $z = 0.5 + {\tt i}  t, t\in [0, 50]$;   (middle)  errors  associated with the sampled nodes in $z = 0.5 + {\tt i}  t, t\in [-50, 50]$;  (right) the phase portrait  of $\xi(z)\approx \zeta(z)$ in a striped region.}
\label{figure7_2}
\end{figure}

\end{example}

\begin{example} \label{egadd:lne0}
{
We  apply {\sf b-d-Lawson}  to solve the following 
Laplace problem \cite[Section 24]{natr:2026} 
\begin{equation}\label{eq:laplace}
\begin{cases}
\Delta u(z) = 0, & z\in\Omega\subset \bbR^2,\\
u(z) = h(z), & z\in \Gamma=\partial \Omega,
\end{cases}
\end{equation}
where $h:\bbR^2\rightarrow \bbR$ is a given real continuous boundary function. Following  \cite[Section 24]{natr:2026}, we choose the parameterized boundary curve as ($\bbC\cong \bbR^2$)
\[ \Gamma = (1+ \sin(4t)/4)
\exp({\tt i} t),\quad 0\leq t < 2\pi. \]
It is known that the solution $u$ of  \eqref{eq:laplace} is harmonic and is the real part $\Re(f(z))$ of a certain analytic function
$f(z),$ i.e., $f(z) = u(z) + {\tt i}v(z), \; z\in \Omega$, where $v$ is its conjugate harmonic function on $\Omega$. }

{In \cite[Section 24]{natr:2026}, a  AAA least-squares approach  is proposed which consists of the following steps: 
\begin{enumerate}[(1)]
\item  Approximate $h$ by a complex rational function $\xi\in \scrR_{(n)}$ on $\Gamma$;
\item  Compute the poles of $\xi$, and discard those inside $\Gamma;$
\item Use linear least-squares to fit $h$ by the remaining poles outside $\Gamma.$
\end{enumerate}
In Step (2), suppose $\bz=[z_1,\dots,z_k]^{\T}\in \bbC^k$ contains the poles $z_j$ of $\xi\in \scrR_{(n)}$ outside $\Omega$ with their minimal distances ${d_j}$ to the boundary: $d_j=\min|z_j-\Gamma|$.  Then a rational function $r$ of the form 
\begin{equation}\label{eq:Laplacer}
r(z)=c_0+\sum_{j=1}^{k} \left(\frac{d_j c^{(1)}_j}{z-z_j}+\frac{d_j^2 c^{(2)}_j}{(z-z_j)^2}\right)
\end{equation}
is constructed in Step (3) as an analytic function to approximate $f(z)$. The coefficients $c_0$ and $\{(c_j^{(1)}, c_j^{(2)})\}_{j=1}^k$ are computed by a least-squares problem with boundary sampled points $\{x_j\}_{j=1}^m$ to minimize $\sum_{j=1}^m|\Re(r(x_j))-h(x_j)|^2$. For a more detailed discussion, see \cite[Section 24]{natr:2026}.}

In this framework, we refer to the algorithm employing AAA to compute the rational approximation $\xi$ in Step (1) as {\sf AAA-LS}, whereas the implementation using {\sf b-d-Lawson}(40) is denoted as {\sf b-d-Lawson-LS}.
Specifically, for {\sf b-d-Lawson-LS}, we enforce four interpolation constraints (highlighted in red in the (1,1)-subfigure of Figure \ref{figure_add}) when computing $\xi$ in Step (1). The rationale behind these constraints is that these points lie in a region where $f(z)$ admits a significant analytic extension (see \cite[Figure 24.2]{natr:2026}), meaning they typically exhibit better smoothness or analytic properties. This helps ensure that the rational approximation provides more accurate boundary conditions. 
In cases where certain points in the boundary condition carry greater significance—for instance, when dealing with noisy data—such interpolation constraints can similarly be imposed. 

Under these settings, we numerically solve problem \eqref{eq:laplace} using \textsf{b-d-Lawson-LS} and conduct accuracy comparisons with \textsf{AAA-LS} across different approximation degrees $n$. For our computational experiments, we employ $2000$ equally distributed discretization points along $t\in [0,2\pi]$ while evaluating the solution accuracy on refined nodes with $10000$ points. Since $r(z)$ in \eqref{eq:Laplacer} is pole-free in $\Omega$, it is analytic therein, ensuring that its real part $u(z) = \Re(r(z))$ is harmonic and naturally satisfies
$$\Delta u(z) = 0, \quad z \in \Omega.$$ Consequently, the accuracy of $u(z)$ for solving the Laplace equation \eqref{eq:laplace} reduces to how well it approximates the boundary condition $u(z) = h(z)$ on $\Gamma = \partial \Omega$. 
The comparative performance results are presented visually in Figure \ref{figure_add}, demonstrating the boundary condition approximation effectiveness (maximum error) of both methods.
\begin{figure}[h!!!]
\begin{center}
{\includegraphics[width= 5.9in ]{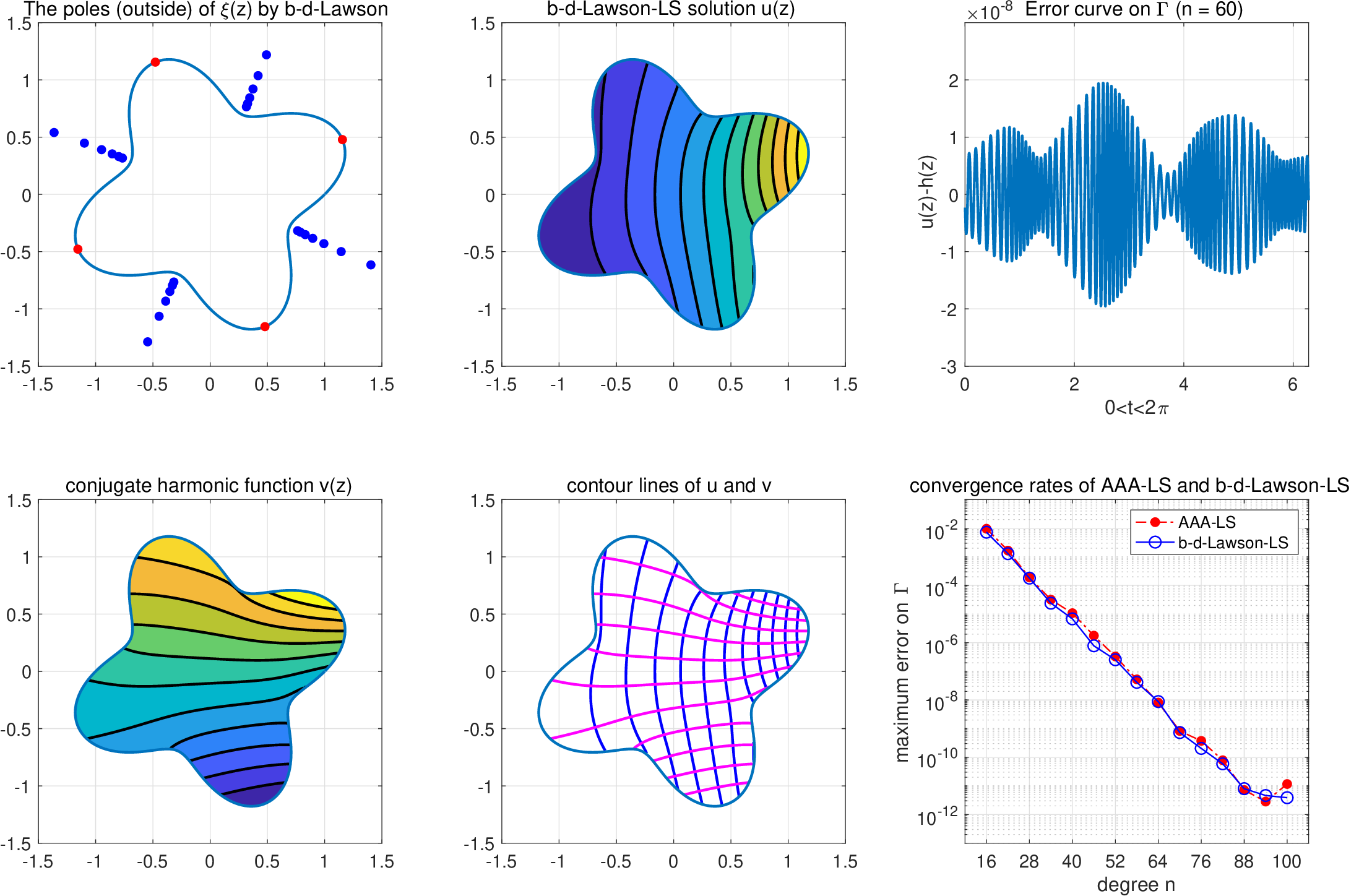} }
\caption{\small 
{
Illustration of the {\sf b-d-Lawson-LS} method applied to a Laplace Dirichlet problem on a smooth nonconvex domain $\Omega$, with boundary data $h(z) = \exp(\Re(z))$. First, $h$ is approximated by a complex rational function $\xi \in \scrR_{(n)}$  for $n=60$ using {\sf b-d-Lawson} with four interpolation constraints (red); the poles (initially placed both inside and outside $\Omega$) resulting from  {\sf b-d-Lawson} are restricted to exterior poles (blue dots, the (1,1)-subfigure) after discarding interior ones. These poles then anchor a least-squares fit to $u(z)$ via $\Re(f(z))$, where $f(z) = u(z) +{\tt  i}v(z)$ is analytic in $\Omega$. The (1,1)-subfigure displays $\Gamma=\partial\Omega$, exterior poles (blue), and interpolation constraints (red) for {\sf b-d-Lawson}; the (1,2)-subfigure shows the resulting solution $u$; the (1,3)-subfigure plots the error $u|_\Gamma - h$; the (2,1)-subfigure visualizes the conjugate harmonic function $v(z) = \Im(f(z))$; the (2,2)-subfigure depicts the orthogonal level curves of $u$ and $v$, and the (2,3)-subfigure demonstrates exponential convergence of {\sf b-d-Lawson-LS} w.r.t. the degree $n$.
}
}
\label{figure_add}
\end{center}
\end{figure}

\end{example}


\begin{example} \label{eg5:lne0}
 In our final example, we approximate 
\[ \sign_{\mathbf{E} /\mathbf{F}}(x) = \begin{cases}
-1, & x\in \mathbf{E}\subset \bbC\\
1, & x \in \mathbf{F}\subset \bbC
\end{cases}
\]
by   $\xi \in \scrR_{(15)}$ where $ \mathbf{E}\cap \mathbf{F} = \emptyset$.  The problem of finding the minimax rational approximation $\xi$ over $\mathbf{E}\cup \mathbf{F}$ is known as the  Zolotarev sign problem (Problem Z4) and was recently considered in   \cite{trwi:2024} with the help of the  AAA method. As exploited in \cite{trwi:2024}, in the \texttt{Chebfun} AAA code {\tt aaa.m}, a modification was introduced as an option in July 2024 specified by an optional flag `sign' for this special problem Z4.  Choose $\mathbf{E} = \left\{  -3+\texttt{i} \cos\big(\frac{j\pi}{200}\big) : 0\leq j\leq 200 \right\}$  and let $\mathbf{F}$ be a set containing 2000 equally spaced points on the unit circle. Applying AAA,  AAA-Lawson(40) and {\sf b-d-Lawson}(40), we first plot the corresponding error curves in the first row in Figure \ref{figure7_4}. For each method, the circle (in black) with the radius $e(\xi)$ is also plotted. In particular,   the (1,3)-subfigure of Figure \ref{figure7_4} is from a plain AAA method without specifying the optional flag `sign', while the (1,4)-subfigure is from AAA-Lawson(40) with the option `sign'. One can observe that {\sf b-d-Lawson} can produce more accurate approximant for this example. 

To verify  {\sf b-d-Lawson} with interpolation conditions furthermore, we next fix two ending points of $\mathbf{E}$  as interpolation nodes for {\sf b-d-Lawson}(40), and present the error curve from  {\sf b-d-Lawson}(40) in the (2,3)-subfigure. Due to the interpolation conditions, we noticed that the maximum error $e(\xi)$ is larger than the one in the (1,2)-subfigure.  
 
\begin{figure}[h!!!]
\begin{center}
{\includegraphics[width= 5.8in, height  = 1.4in ]{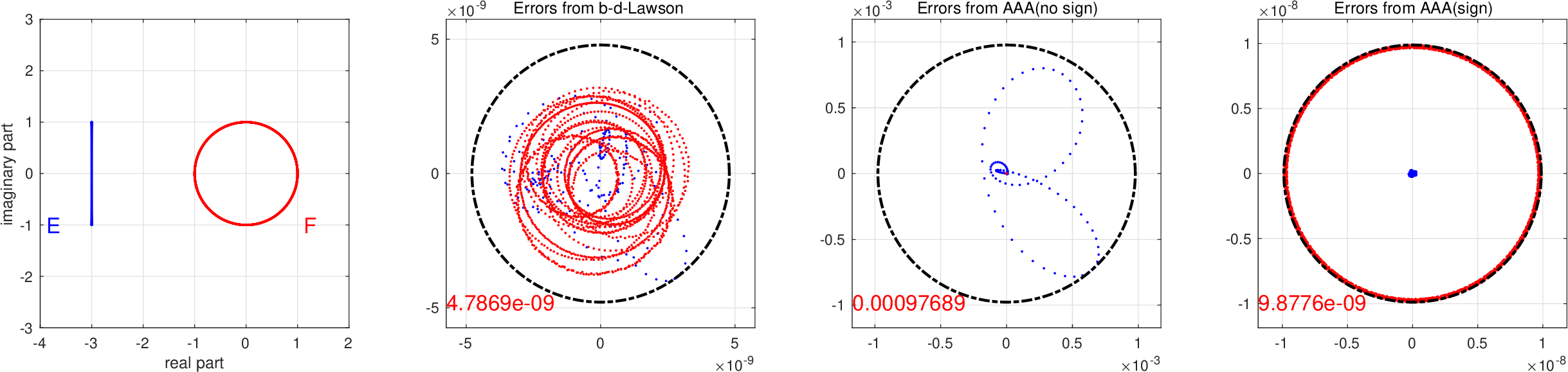}  
\vskip 15mm
\includegraphics[width= 2.8in, height  = 1.4in ]{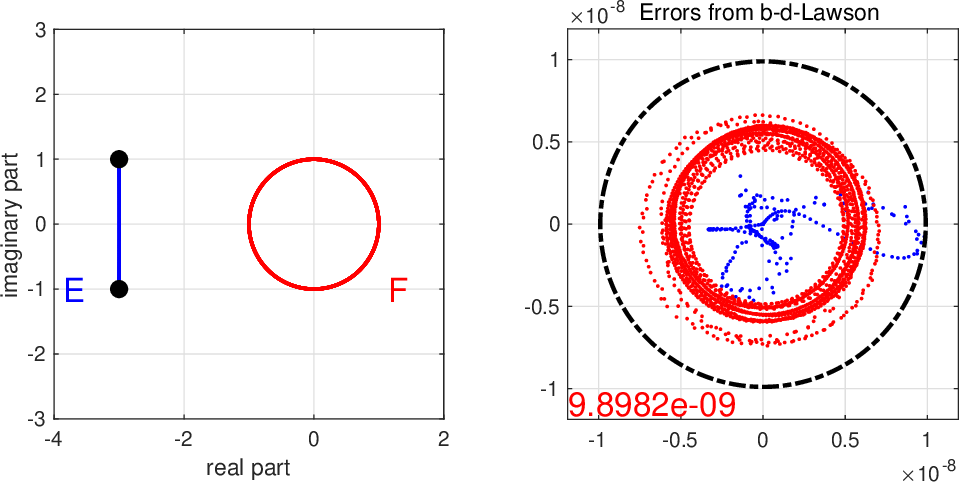} }
\caption{\small The (1,1)-subfigure: the sets $\mathbf{E}$ and $\mathbf{F}$;  the (1,2)-subfigure: errors associated with the samples from {\sf b-d-Lawson}($40$);  the (1,3)-subfigure: errors associated with the samples from  AAA without the option `sign';   the (1,4)-subfigure: errors associated with the samples from AAA-Lawson(40)   with   the option `sign'; the (2,2)-subfigure: the sets $\mathbf{E}$ and $\mathbf{F}$ with two interpolation nodes; the (2,3)-subfigure: errors associated with the samples from {\sf b-d-Lawson}($40$) subject to two interpolation conditions. For each method, the circle (in black) with the radius $e(\xi)$ is also plotted for $\xi\in \scrR_{(15)}$.}
\label{figure7_4}
\end{center}
\end{figure}
\end{example}

\section{Conclusions and further issues}\label{sec:conclusion}
This paper offers an extension of the recent dual-based development on the standard discrete rational minimax approximation \cite{zhyy:2025}. First, the proposed method {\sf b-d-Lawson} is a new implementation of {\sf d-Lawson} \cite{zhyy:2025} based on the barycentric representation, and can compute the minimax approximant $\xi\in \scrR_{(n)}$ for the given   data $\{(x_j,f_j)\}_{j=1}^m$. Since the barycentric representation is a good parameterization to meet the interpolation conditions, it can be incorporated into the dual framework of the interpolation constrained rational minimax problem \eqref{eq:baryInterpMinmax}. A dual problem of \eqref{eq:baryInterpMinmax} is established and properties of weak duality \eqref{eq:weakduality} and strong duality \eqref{eq:strongdual}  are discussed. Moreover, in this dual framework, a sufficient condition \eqref{eq:strongdualityRuttan} to achieve the minimax solution of \eqref{eq:baryInterpMinmax} is presented in Theorem \ref{thm:strongdualityRuttan}. A few numerical examples are tested to demonstrate the performance of {\sf b-d-Lawson}.  

We conclude this paper by mentioning a few issues that deserve further investigation. It should be noted that {\sf b-d-Lawson} cannot be guaranteed to handle \eqref{eq:baryInterpMinmax} in all cases. In fact, as we have mentioned that only a sufficient condition \eqref{eq:strongdualityRuttan} for the minimax approximant is provided in Theorem \ref{thm:strongdualityRuttan} under the assumption:  $\wtd\beta_j\ne 0~(1\le j\le \ell)$  in  the pair $(\wtd p,\wtd q)$ that achieves $d(\bw^*)$. Although in numerical results, the latter assumption usually holds  for a general function $f\not\equiv 0$, it fails for the constant function $f\equiv 0$. A remedy for this is to impose new constraints (e.g., $\sum_{j=1}^\ell|\wtd \beta_j|^2=1$, or $\wtd\beta_j\ne 0~(1\le j\le \ell)$) into \eqref{eq:linearminmax2}. This will lead to  new dual problems but will significantly complicate the computation of the associated dual functions. Another issue that has not been treated for {\sf b-d-Lawson} is its convergence: even though for $\ell=0$, the convergence analysis in \cite{zhha:2025} can be analogously applied, the general case with $\ell>0$ deserves a further  investigation.

\section*{Acknowledgments}
We would like to thank the anonymous referees for their constructive comments and suggestions that have improved the overall presentation and clarity of the paper.
 
{\small
\def\noopsort#1{}\def\l{\char32l}\def\v#1{{\accent20 #1}}
  \let\^^_=\v\def\hbk{hardback}\def\pbk{paperback}
\providecommand{\href}[2]{#2}
\providecommand{\arxiv}[1]{\href{http://arxiv.org/abs/#1}{arXiv:#1}}
\providecommand{\url}[1]{\texttt{#1}}
\providecommand{\urlprefix}{URL }

}

\begin{thebibliography}{10}

\bibitem{alyo:1983}
\newblock A.~Allison and N.~Young,
\newblock Numerical algorithms for the {Nevanlinna-Pick} problem,
\newblock \emph{Numer. Math.}, \textbf{42} (1983), 125--145.


\bibitem{angp:2025}
{\newblock A.~C. Antoulas, I.~V. Gosea and C.~Poussot-Vassal,
\newblock On the {L}oewner framework, the {K}olmogorov superposition theorem,
  and the curse of dimensionality,
\newblock \emph{SIAM Rev.}, \textbf{67} (2025), 737--770,
\newblock \urlprefix\url{https://doi.org/10.1137/24M1656657}.}

\bibitem{bapr:1972}
\newblock I.~Barrodale, M.~J.~D. Powell and F.~D.~K. Roberts,
\newblock The differential correction algorithm for rational $\ell_\infty$
  approximation,
\newblock \emph{SIAM J. Numer. Anal.}, \textbf{9} (1972), 493--504.

\bibitem{beck:2000}
\newblock B.~Beckermann,
\newblock {The condition number of real Vandermonde, Krylov and positive
  definite Hankel matrices},
\newblock \emph{Numer. Math.}, \textbf{85} (2000), 553--577.

\bibitem{begu:2017}
\newblock M.~Berljafa and S.~G\"{u}ttel,
\newblock The {RKFIT} algorithm for nonlinear rational approximation,
\newblock \emph{SIAM J. Sci. Comput.}, \textbf{39} (2017), A2049--A2071,
\newblock \urlprefix\url{https://doi.org/10.1137/15M1025426}.

\bibitem{bemi:1997}
\newblock J.-P. Berrut and H.~D. Mittelmann,
\newblock Matrices for the direct determination of the barycentric weights of
  rational interpolation,
\newblock \emph{J. Comput. Appl. Math.}, \textbf{78} (1997), 355--370.

\bibitem{bemi:1997b}
\newblock J.-P. Berrut and H.~Mittelmann,
\newblock Lebesgue constant minimizing linear rational interpolation of
  continuous functions over the interval,
\newblock \emph{Computers Math. Appl.}, \textbf{33} (1997), 77--86.

\bibitem{betr:2004}
\newblock J.-P. Berrut and L.~N. Trefethen,
\newblock Barycentric {L}agrange interpolation,
\newblock \emph{SIAM Rev.}, \textbf{46} (2004), 501--517.

\bibitem{bekl:2014}
\newblock J.-P. Berrut and G.~Klein,
\newblock Recent advances in linear barycentric rational interpolation,
\newblock \emph{J. Comput. Appl. Math.}, \textbf{259} (2014), 95--107.

\bibitem{boyd:2004}
\newblock S.~Boyd and L.~Vandenberghe,
\newblock \emph{Convex Optimization},
\newblock Cambridge University Press, 2004.

\bibitem{brnt:2021}
\newblock P.~D. Brubeck, Y.~Nakatsukasa and L.~N. Trefethen,
\newblock Vandermonde with {A}rnoldi,
\newblock \emph{SIAM Rev.}, \textbf{63} (2021), 405--415.

\bibitem{bygl:2001}
\newblock C.~Byrnes, T.~Georgiou and A.~Lindquist,
\newblock A generalized entropy criterion for {Nevanlinna-Pick} interpolation
  with degree constraint,
\newblock \emph{IEEE T. Automat. Contr.}, \textbf{46} (2001), 822--839.

\bibitem{chlo:1963}
\newblock E.~W. Cheney and H.~L. Loeb,
\newblock Two new algorithms for rational approximation,
\newblock \emph{Numer. Math.}, \textbf{3} (1961), 72--75.

\bibitem{chli:2000}
\newblock W.~Cheney and W.~Light,
\newblock \emph{A Course in Approximation Theory},
\newblock Pacific Grove, CA: Brooks/Cole, 2000.

\bibitem{clin:1972}
\newblock A.~K. Cline,
\newblock Rate of convergence of {L}awson's algorithm,
\newblock \emph{Math. Comp.}, \textbf{26} (1972), 167--176.

\bibitem{cops:2002}
\newblock C.~Coelho, J.~Phillips and L.~Silveira,
\newblock Passive constrained rational approximation algorithm using
  nevanlinna-pick interpolation,
\newblock in \emph{Proceedings 2002 Design, Automation and Test in Europe
  Conference and Exhibition}, 2002,
\newblock 923--930.

\bibitem{drnt:2024}
\newblock T.~A. Driscoll, Y.~Nakatsukasa and L.~N. Trefethen,
\newblock {AAA} rational approximation on a continuum,
\newblock \emph{SIAM J. Sci. Comput.}, \textbf{46} (2024), A929--A952.


\bibitem{ehle:1976}
{\newblock B.~L. Ehle,
\newblock On certain order constrained {C}hebyshev rational approximations,
\newblock \emph{J. Approx. Theory}, \textbf{17} (1976), 297--306.}


\bibitem{elli:1978}
\newblock G.~H. Elliott,
\newblock \emph{The construction of {Chebyshev} approximations in the complex
  plane},
\newblock PhD thesis, Faculty of Science (Mathematics), University of London,
  1978.




\bibitem{fint:2018}
\newblock S.-I. Filip, Y.~Nakatsukasa, L.~N. Trefethen and B.~Beckermann,
\newblock Rational minimax approximation via adaptive barycentric
  representations,
\newblock \emph{SIAM J. Sci. Comput.}, \textbf{40} (2018), A2427--A2455.

\bibitem{fifr:1990}
\newblock B.~Fischer and R.~Freund,
\newblock On the constrained {C}hebyshev approximation problem on ellipses,
\newblock \emph{J. Approx. Theory}, \textbf{62} (1990), 297--315.

\bibitem{flho:2007}
\newblock M.~S. Floater and K.~Hormann,
\newblock Barycentric rational interpolation with no poles and high rates of
  approximation,
\newblock \emph{Numer. Math.}, \textbf{107} (2007), 315--331.

\bibitem{govl:2013}
\newblock G.~H. Golub and C.~F. {Van Loan},
\newblock \emph{Matrix Computations},
\newblock 4th edition,
\newblock Johns Hopkins University Press, Baltimore, Maryland, 2013.

\bibitem{gogu:2021}
\newblock I.~V. Gosea and S.~G\"{u}ttel,
\newblock Algorithms for the rational approximation of matrix-valued functions,
\newblock \emph{SIAM J. Sci. Comput.}, \textbf{43} (2021), A3033--A3054,
\newblock \urlprefix\url{https://doi.org/10.1137/20M1324727}.

\bibitem{guse:1999}
\newblock B.~Gustavsen and A.~Semlyen,
\newblock Rational approximation of frequency domain responses by vector
  fitting,
\newblock \emph{IEEE Trans. Power Deliv.}, \textbf{14} (1999), 1052--1061.

\bibitem{gutk:1983}
\newblock M.~H. Gutknecht,
\newblock On complex rational approximation. {Part I}: The characterization
  problem,
\newblock in \emph{{Computational Aspects of Complex Analysis (H. Werneret
  at.,eds.). Dordrecht : Reidel}}, 1983,
\newblock 79--101.

\bibitem{gukl:2012}
\newblock S.~G\"{u}ttel and G.~Klein,
\newblock Convergence of linear barycentric rational interpolation for analytic
  functions,
\newblock \emph{SIAM J. Numer. Anal.}, \textbf{50} (2012), 2560--2580,
\newblock \urlprefix\url{https://doi.org/10.1137/120864787}.

\bibitem{henr:1962}
\newblock P.~Henrici,
\newblock \emph{Elements of Numerical Analysis},
\newblock Wiley, New York, 1962.

\bibitem{henr:1979}
\newblock P.~Henrici,
\newblock Barycentric formulas for interpolating trigonometric polynomials and
  their conjugates,
\newblock \emph{Numer. Math.}, \textbf{33} (1979), 225--234.

\bibitem{hesp:1991}
\newblock G.~Herglotz, I.~Schur, G.~Pick, R.~Nevanlinna and H.~Weyl,
\newblock \emph{Ausgew{\"a}hlte arbeiten zu den urspr{\"u}ngen der
  {S}chur-analysis}, vol.~16 of Teubner-Archiv zur Mathematik,
\newblock B.G. Teubner Verlagsgesellschaft mbH, Stuttgart, 1991.

\bibitem{high:2002}
\newblock N.~J. Higham,
\newblock \emph{Accuracy and Stability of Numerical Algorithms, Second
  edition},
\newblock SIAM, Philadephia, USA, 2002.

\bibitem{high:2004a}
\newblock N.~J. Higham,
\newblock The numerical stability of barycentric {L}agrange interpolation,
\newblock \emph{IMA J. Numer. Anal.}, \textbf{24} (2004), 547--556.

\bibitem{hoka:2020}
\newblock J.~M. Hokanson,
\newblock Multivariate rational approximation using a stabilized
  {S}anathanan-{K}oerner iteration, 2020,
\newblock \urlprefix\url{arXiv:2009.10803v1}.

\bibitem{ioni:2013}
\newblock A.~C. Ioni\c{t}\u{a},
\newblock \emph{Lagrange Rational Interpolation and its Applications to
  Approximation of Large-Scale Dynamical Systems},
\newblock PhD thesis, Rice University, USA, 2013.

\bibitem{kali:2008}
\newblock J.~Karlsson and A.~Lindquist,
\newblock Stability-preserving rational approximation subject to interpolation
  constraints,
\newblock \emph{IEEE T. Automat. Contr.}, \textbf{53} (2008), 1724--1730.

\bibitem{laws:1961}
\newblock C.~L. Lawson,
\newblock \emph{Contributions to the Theory of Linear Least Maximum
  Approximations},
\newblock PhD thesis, UCLA, USA, 1961.

\bibitem{lixu:1990}
\newblock J.-K. Li and G.-L. Xu,
\newblock On the problem of best rational approximation with interpolating
  constraints {(I)},
\newblock \emph{J. Comput. Math.}, \textbf{8} (1990), 233--240,
\newblock
  \urlprefix\url{http://global-sci.org/intro/article_detail/jcm/9436.html}.

\bibitem{lixu:1990b}
\newblock J.-K. Li and G.-L. Xu,
\newblock On the problem of best rational approximation with interpolating
  constraints {(II)},
\newblock \emph{J. Comput. Math.}, \textbf{8} (1990), 241--251,
\newblock
  \urlprefix\url{http://global-sci.org/intro/article_detail/jcm/9437.html}.

\bibitem{li:2008b}
\newblock R.-C. Li,
\newblock {Vandermonde} matrices with {Chebyshev} nodes,
\newblock \emph{Linear Algebra Appl.}, \textbf{428} (2008), 1803--1832.

\bibitem{lilb:2013}
\newblock X.~Liang, R.-C. Li and Z.~Bai,
\newblock Trace minimization principles for positive semi-definite pencils,
\newblock \emph{Linear Algebra Appl.}, \textbf{438} (2013), 3085--3106.

\bibitem{limp:2022}
\newblock P.~Lietaert, K.~Meerbergen, J.~P\'erez and B.~Vandereycken,
\newblock Automatic rational approximation and linearization of nonlinear
  eigenvalue problems,
\newblock \emph{IMA J. Numer. Anal.}, \textbf{42} (2022), 1087--1115.

\bibitem{loeb:1957}
\newblock H.~L. Loeb,
\newblock \emph{On rational fraction approximations at discrete points},
\newblock Technical report, Convair Astronautics, 1957,
\newblock Math. Preprint \#9.

\bibitem{lubr:2011}
\newblock L.~A. Luxemburg and P.~R. Brown,
\newblock The scalar {Nevanlinna--Pick} interpolation problem with boundary
  conditions,
\newblock \emph{J. Comput. Appl. Math.}, \textbf{235} (2011), 2615--2625,
\newblock
  \urlprefix\url{https://www.sciencedirect.com/science/article/pii/S0377042710006308}.

\bibitem{luws:2020}
\newblock L.~Monz{\'o}n, W.~Johns, S.~Iyengar, M.~Reynolds, J.~Maack and
  K.~Prabakar,
\newblock A multi-function {AAA} algorithm applied to frequency dependent line
  modeling,
\newblock in \emph{2020 IEEE Power \& Energy Society General Meeting (PESGM)},
  2020,
\newblock 1--5.

\bibitem{nase:2018}
\newblock Y.~Nakatsukasa, O.~S\`ete and L.~N. Trefethen,
\newblock The {AAA} algorithm for rational approximation,
\newblock \emph{SIAM J. Sci. Comput.}, \textbf{40} (2018), A1494--A1522.

\bibitem{natr:2020}
\newblock Y.~Nakatsukasa and L.~N. Trefethen,
\newblock An algorithm for real and complex rational minimax approximation,
\newblock \emph{SIAM J. Sci. Comput.}, \textbf{42} (2020), A3157--A3179.


\bibitem{natr:2026}
{\newblock Y.~Nakatsukasa and L.~N. Trefethen,
\newblock Applications of {AAA} rational approximation,
\newblock \emph{Acta Numer.}, to appear,
\newblock \urlprefix\url{https://arxiv.org/abs/2510.16237}.}



\bibitem{nast:2023}
\newblock Y.~Nakatsukasa, O.~S\`ete and L.~N. Trefethen,
\newblock The first five years of the {AAA} algorithm,
\newblock \emph{Intl. Cong. Basic Sci.},
\newblock \urlprefix\url{https://arxiv.org/pdf/2312.03565},
\newblock To appear.

\bibitem{neva:1919}
\newblock R.~Nevanlinna,
\newblock \"{U}ber beschr\"ankte funktionen, die in gegebenen punkten
  vorgeschrieben werte annehmen,
\newblock \emph{Ann. Acad. Sci. Fenn. Sel A.}, \textbf{13} (1919), 1--72.

\bibitem{nowr:2006}
\newblock J.~Nocedal and S.~Wright,
\newblock \emph{Numerical Optimization},
\newblock 2nd edition,
\newblock Springer, New York, 2006.

\bibitem{patr:2009}
\newblock R.~Pach{\'o}n and L.~N. Trefethen,
\newblock {Barycentric-Remez} algorithms for best polynomial approximation in
  the chebfun system,
\newblock \emph{BIT}, \textbf{49} (2009), 721--741.

\bibitem{pan:2016}
\newblock V.~Y. Pan,
\newblock How bad are {V}andermonde matrices?,
\newblock \emph{SIAM J. Matrix Anal. Appl.}, \textbf{37} (2016), 676--694.

\bibitem{pick1916}
\newblock G.~Pick,
\newblock \"{U}ber die beschr\"ankungen analytischer funktionen, welche durch
  vorgegebene funktionswerte bewirkt werden,
\newblock \emph{Math. Ann.}, \textbf{77} (1916), 7--23.

\bibitem{rish:1961}
\newblock T.~J. Rivlin and H.~S. Shapiro,
\newblock A unified approach to certain problems of approximation and
  minimization,
\newblock \emph{J. Soc. Indust. Appl. Math.}, \textbf{9} (1961), 670--699.

\bibitem{ruti:1976}
\newblock H.~Rutishauser,
\newblock \emph{Vorlesungen \"uber numerische Mathematik, Vol. 1, Birkh\"auser,
  Basel, Stuttgart, 1976; English translation, Lectures on Numerical
  Mathematics},
\newblock Walter Gautschi, ed., Birkh\"auser, Boston, 1976.

\bibitem{rutt:1985}
\newblock A.~Ruttan,
\newblock A characterization of best complex rational approximants in a
  fundamental case,
\newblock \emph{Constr. Approx.}, \textbf{1} (1985), 287--296.

\bibitem{saad:2003}
\newblock Y.~Saad,
\newblock \emph{Iterative Methods for Sparse Linear Systems},
\newblock 2nd edition,
\newblock SIAM, Philadelphia, 2003.

\bibitem{sava:1977}
\newblock E.~B. Saff and R.~S. Varga,
\newblock Nonuniqueness of best approximating complex rational functions,
\newblock \emph{Bull. Amer. Math. Soc.}, \textbf{83} (1977), 375--377.

\bibitem{sako:1963}
\newblock C.~K. Sanathanan and J.~Koerner,
\newblock Transfer function synthesis as a ratio of two complex polynomials,
\newblock \emph{IEEE T. Automat. Contr.}, \textbf{8} (1963), 56--58,
\newblock \urlprefix\url{https://doi.org/10.1109/TAC.1963.1105517}.

\bibitem{scwe:1986}
\newblock C.~Schneider and W.~Werner,
\newblock Some new aspects of rational interpolation,
\newblock \emph{Math. Comp.}, \textbf{47} (1986), 285--299.

\bibitem{tawi:1974}
\newblock G.~D. Taylor and J.~William,
\newblock Existence questions of {C}hebyshev by interpolating for the problem
  approximation rationals,
\newblock \emph{Math. Comp.}, \textbf{28} (1974), 1097--1103.

\bibitem{tayl:1945}
\newblock W.~J. Taylor,
\newblock Method of {L}agrangian curvilinear interpolation,
\newblock \emph{J. Res. Nat. Bur. Standards}, \textbf{35} (1945), 151--155.

\bibitem{this:1993}
\newblock J.-P. Thiran and M.-P. Istace,
\newblock Optimality and uniqueness conditions in complex rational {C}hebyshev
  approximation with examples,
\newblock \emph{Constr. Approx.}, \textbf{9} (1993), 83--103.

\bibitem{tref:2019a}
\newblock L.~N. Trefethen,
\newblock \emph{{Approximation Theory and Approximation Practice, {E}xtended
  Edition}},
\newblock SIAM, 2019.

\bibitem{trwi:2024}
\newblock L.~N. Trefethen and H.~D. Wilber,
\newblock Computation of {Z}olotarev rational functions, 2024,
\newblock \urlprefix\url{https://arxiv.org/abs/2408.14092}.

\bibitem{wals:1932}
\newblock J.~L. Walsh,
\newblock On interpolation and approximation by rational functions with
  preassigned poles,
\newblock \emph{Trans. Amer. Math. Soc.}, \textbf{34} (1932), 22--74.

\bibitem{wang:1980}
{\newblock R.~Wang,
\newblock \emph{Numerical Rational Approximation (in Chinese)},
\newblock Shanghai Science-Technology Press, Shanghai, 1980.}

\bibitem{wazh:2004}
{\newblock R.~Wang and G.~Zhu,
\newblock \emph{Rational Function Approximations and Their Applications (in
  Chinese)}, vol.~31 of Series in Information and Computational Sciences,
\newblock Science Press, Beijing, 2004.}

\bibitem{wern:1984}
\newblock W.~Werner,
\newblock Polynomial interpolation: {L}agrange versus {N}ewton,
\newblock \emph{Math. Comp.}, \textbf{43} (1984), 205--217.

\bibitem{will:1972}
\newblock J.~Williams,
\newblock Numerical {C}hebyshev approximation in the complex plane,
\newblock \emph{SIAM J. Numer. Anal.}, \textbf{9} (1972), 638--649.

\bibitem{will:1979}
\newblock J.~Williams,
\newblock Characterization and computation of rational {C}hebyshev
  approximations in the complex plane,
\newblock \emph{SIAM J. Numer. Anal.}, \textbf{16} (1979), 819--827.

\bibitem{wulb:1980}
\newblock D.~E. Wulbert,
\newblock On the characterization of complex rational approximations,
\newblock \emph{Illinois J. Math.}, \textbf{24} (1980), 140--155.

\bibitem{wuyt:1974}
\newblock L.~Wuytack,
\newblock On some aspects of the rational interpolation problem,
\newblock \emph{SIAM J. Numer. Anal.}, \textbf{11} (1974), 52--60.

\bibitem{yazz:2023}
\newblock L.~Yang, L.-H. Zhang and Y.~Zhang,
\newblock The {L}q-weighted dual programming of the linear {C}hebyshev
  approximation and an interior-point method,
\newblock \emph{Adv. Comput. Math.}, 50:80 (2024).



\bibitem{zhan:2026}
{\newblock L.-H. Zhang,
\newblock Optimality conditions for rational minimax approximations: Bridging
  {R}uttan's criteria to dual-based methods, 2026,
\newblock \urlprefix\url{https://arxiv.org/pdf/2602.07862v1}}

\bibitem{zhha:2025}
\newblock L.-H. Zhang and S.~Han,
\newblock A convergence analysis of {L}awson's iteration for computing
  polynomial and rational minimax approximations,
{\newblock \emph{SIAM J. Numer. Anal.}, \textbf{63} (2025), 2249--2271.} 

\bibitem{zhsl:2024}
\newblock L.-H. Zhang, Y.~Su and R.-C. Li,
\newblock Accurate polynomial fitting and evaluation via {A}rnoldi,
\newblock \emph{Numerical Algebra, Control and Optimization}, \textbf{14}
  (2024), 526--546.

\bibitem{zhyy:2025}
\newblock L.-H. Zhang, L.~Yang, W.~H. Yang and Y.-N. Zhang,
\newblock A convex dual problem for the rational minimax approximation and
  {L}awson's iteration,
\newblock \emph{Math. Comp.}, \textbf{94} (2025), 2457--2494,
\newblock DOI: https://doi.org/10.1090/mcom/4021.

\bibitem{zhzz:2025}
\newblock L.-H. Zhang, Y.-N. Zhang, C.~Zhang and S.~Han,
\newblock Rational minimax approximation of matrix-valued functions, 2025,
\newblock \urlprefix\url{https://arxiv.org/pdf/2508.06378v2}.

\end{thebibliography}
\end{document}